\documentclass{article}
\usepackage[utf8]{inputenc}
\usepackage[left=3cm,right=3cm,top=2cm,bottom=2cm]{geometry}

\usepackage{comment}
\usepackage{caption}
\usepackage{subcaption}
\usepackage{graphicx}
\usepackage{multirow}

\usepackage{amsfonts,amsmath,amssymb,amsthm}
\usepackage{dsfont}

\usepackage[dvipsnames]{xcolor}

\usepackage{pgfplots}

\usepackage{hyperref}

\usepackage[normalem]{ulem}

\newtheorem{theorem}{Theorem}
\newtheorem{corollary}[theorem]{Corollary}
\newtheorem{lemma}[theorem]{Lemma}
\newtheorem{definition}[theorem]{Definition}
\newtheorem{proposition}[theorem]{Proposition}
\newtheorem{remark}[theorem]{Remark}

\providecommand{\keywords}[1]
{
  \small	
  \textbf{\textbf{Keywords}:} #1
}

\newcommand{\ee}{\varepsilon}
\newcommand{\EE}{\mathbb{E}}

\newcommand{\PP}{\mathbb{P}}
\newcommand{\RR}{\mathbb{R}}
\newcommand{\A}{\mathcal{A}}
\newcommand{\R}{\mathbb{R}}
\newcommand{\Sym}{\mathbb{S}}

\newcommand{\dint}{\mathrm{d}}

\usepackage{algorithm}

\usepackage{algpseudocode}

\algnewcommand\INPUT{\item[\textbf{Input:}]}%
\algnewcommand\OUTPUT{\item[\textbf{Output:}]}%

\usepackage[flushleft]{threeparttable}
\usepackage{array,booktabs,makecell}
\usepackage[font=small]{caption}

\title{Spectrahedral Regression}
\author{Eliza O'Reilly and Venkat Chandrasekaran}
\date{}

\begin{document}

\maketitle

\begin{abstract}
   Convex regression is the problem of fitting a convex function to a data set consisting of input-output pairs. We present a new approach to this problem called spectrahedral regression, in which we fit a spectrahedral function to the data, i.e. a function that is the maximum eigenvalue of an affine matrix expression of the input. This method represents a significant generalization of polyhedral (also called max-affine) regression, in which a polyhedral function (a maximum of a fixed number of affine functions) is fit to the data. We prove bounds on how well spectrahedral functions can approximate arbitrary convex functions via statistical risk analysis. We also analyze an alternating minimization algorithm for the non-convex optimization problem of fitting the best spectrahedral function to a given data set. We show that this algorithm converges geometrically with high probability to a small ball around the optimal parameter given a good initialization. Finally, we demonstrate the utility of our approach with experiments on synthetic data sets as well as real data arising in applications such as economics and engineering design. 
\end{abstract}



\section{Introduction}

The problem of identifying a function that approximates a given dataset of input-output pairs is a central one in data science.  In this paper we consider the problem of fitting a convex function to such input-output pairs, a task known as \emph{convex regression}.   Concretely, given data $\{x^{(i)}, y^{(i)}\}_{i=1}^n \subset \R^d \times \R$, our objective is to identify a convex function $\hat{f}$ such that $\hat{f}(x^{(i)}) \approx y^{(i)}$ for each $i=1,\dots,n$.  In some applications, one seeks an estimate $\hat{f}$ that is convex and positively homogenous; in such cases, the problem may equivalently be viewed as one of identifying a convex set given (possibly noisy) support function evaluations.   Convex reconstructions in such problems are of interest for several reasons.  First, prior domain information in the context of a particular application might naturally lead a practitioner to seek convex approximations.  One prominent example arises in economics in which the theory of marginal utility implies an underlying convexity relationship.  Another important example arises in computed tomography applications in which one has access to support function evaluations of some underlying set, and the goal is to reconstruct the set; here, due to the nature of the data acquisition mechanism, the set may be assumed to be convex without loss of generality.  A second reason for preferring a convex reconstruction $\hat{f}$ is computational -- in some applications the goal is to subsequently use $\hat{f}$ as an objective or constraint within an optimization formulation.  For example, in aircraft design problems, the precise relationship between various attributes of an aircraft is often not known in closed-form, but input-output data are available from simulations; in such cases, identifying a good convex approximation for the input-output relationship is useful for subsequent aircraft design using convex optimization.

A natural first estimator one might write down is:
\begin{equation} \label{eq:lse}
	\hat{f}^{(n)}_{\mathrm{LSE}} \in {\arg\min}_{f : \R^d \rightarrow \R \text{ is a convex function}} ~~~ \frac{1}{n} \sum_{i=1}^n (y^{(i)} - f(x^{(i)}))^2.
\end{equation}
There always exists a polyhedral function that attains the minimum in \eqref{eq:lse}, and this function may be computed efficiently via convex quadratic programming \cite{PrinceWillsky1990, PrinceWillsky1991, SeijoSen2011}. However, this choice suffers from a number of drawbacks. 
For a large sample size, the quality of the resulting estimate suffers from over-fitting as the complexity of the reconstruction grows with the number of data points.  For small sample sizes, the quality of the resulting estimate is often poor due to noise.  From a statistical perspective, the estimator may also be suboptimal \cite{KurConvexRegression, KurSupport}.  For these reasons, it is of interest to regularize the estimator by considering a suitably constrained class of convex functions.

The most popular approach in the literature to penalize the complexity of the reconstruction in \eqref{eq:lse} is to fit a polyhedral function that is representable as the maximum of at most $m$ affine functions (for a user-specified choice of $m$) to the given data  \cite{Guntuboyina2012, HanWellner2016, HannahDunson2013, MagnaniBoyd2009, Balazs2015}, which is based on the observation that convex functions are suprema of affine functions.  However, this approach is inherently restrictive in situations in which the underlying phenomenon is better modeled by a non-polyhedral convex function, which may not be well-approximated by $m$-polyhedral functions.  Further, in settings in which the estimated function is subsequently used within an optimization formulation, the above approach constrains one to using linear-programming (LP) representable functions.  See Figure \ref{f:3dWage} for a demonstration with economic data. 
\begin{figure}[h!]
     \centering
     \includegraphics[width=\textwidth]{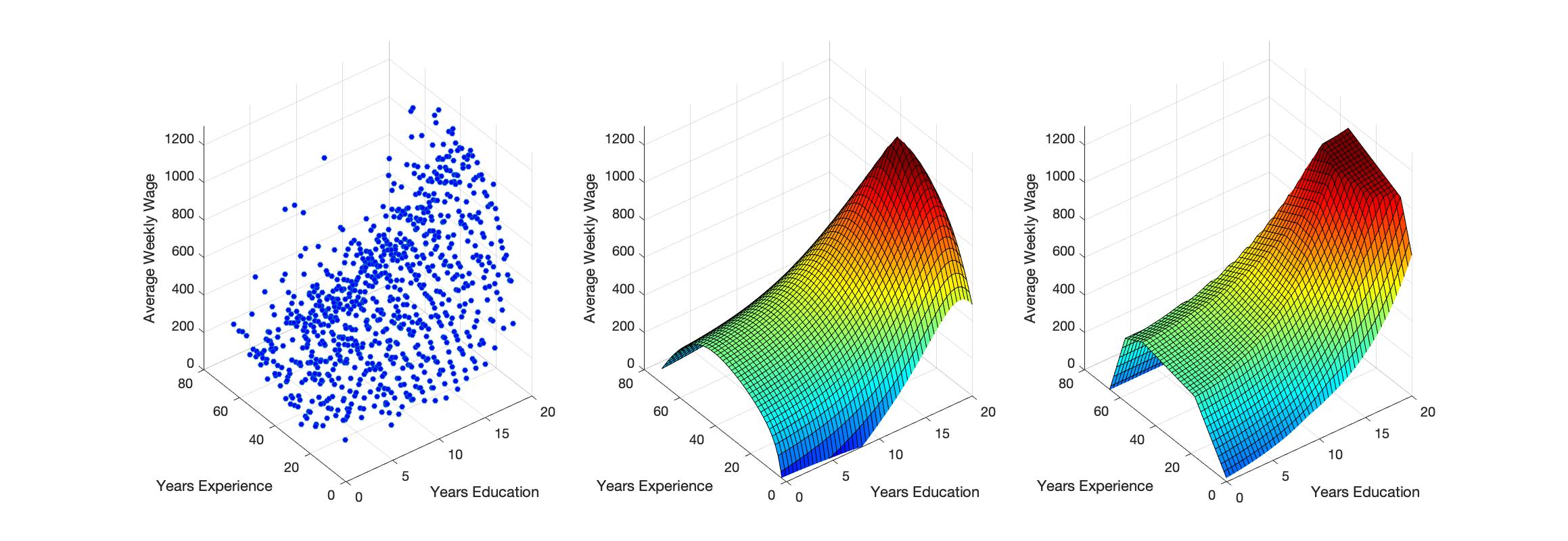}
     \caption{Models for average weekly wage based on years of experience and education using spectrahdedral and polyhedral regression. From left to right: the underlying data set, the spectrahedral ($m= 3$) estimator, and the polyhedral ($m = 6$) estimator. A transformation in the years of education covariate gives a data set that is approximately convex.}
     \label{f:3dWage}
     \vspace{-.2in}
 \end{figure}

To overcome these limitations, we consider fitting spectrahedral functions to data. To define this model class, let $\mathbb{S}^m_k$ denote the set of $m \times m$ real symmetric matrices that are block diagonal with blocks of size at most $k \times k$, with $k$ dividing $m$. 
\begin{definition}\label{def:specfxn}
	Fix positive integers $m, k$ such that $k$ divides $m$. A function $f : \R^d \rightarrow \R$ is called $(m,k)$-spectrahedral if it can be expressed as follows:
	\begin{equation*}
		f(x) = \lambda_{\max}\left(\sum_{i=1}^d A_i x_i + B \right),
	\end{equation*}
	where $A_1, \dots, A_d, B \in \Sym_k^m$. 
	Here $\lambda_{\max}(\cdot)$ is the largest eigenvalue of a matrix.
\end{definition}
An $(m,k)$-spectrahedral function is convex as it is a composition of a convex function with an affine map.  For the case $k=1$, the matrices $A_1,\dots,A_d,B$ are all diagonal and we recover the case of $m$-polyhedral functions.  The case $k=2$ corresponds to second-order-cone-programming (SOCP) representable functions, and the case $k=m$ utilizes the expressive power of semidefinite programming (SDP).  In analogy to the enhanced modeling power of SOCP and SDP in comparison to LP, the class of $(m,k)$-spectrahedral functions is much richer than the set of $m$-polyhedral functions for general $k > 1$. For instance, when $k = 2$ this class contains the function $f(x) = \|x\|_2$. For estimates that are $(m,k)$-spectrahedral, subsequently emplying them within optimization formulations yields optimization problems that can be solved via SOCP and SDP.

An $(m,k)$-spectrahedral function that is positively homogenous (i.e., $B=0$ in the definition above) is the support function of a convex set that is expressible as the linear image of an $(m,k)$-spectraplex defined, for positive integers $k$ and $m$ such that $k$ divides $m$, by
\begin{equation} \label{eq:spectraplex}
 \mathcal{S}_{m,k} = \{M \in \Sym_k^m ~|~ \mathrm{tr}(M) = 1, ~ M \succeq 0\}. 
\end{equation}
We refer to the collection of linear images of $\mathcal{S}_{m,k}$ as {\em $(m,k)$-spectratopes}. Again, the case $k=1$ corresponds to the $m$-simplex, and the corresponding linear images are $m$-polytopes.  Thus, in the positively homogenous case, our proposal is to identify a linear image of an $(m,k)$-spectraplex to fit a given set of support function evaluations.  We note that the case $k=m$ was recently considered in \cite{Soh19}, and we comment in more detail on the comparison between the present paper and \cite{Soh19} in Section \ref{sec:related}.

\subsection{Our Contributions}
We consider the following constrained analog of \eqref{eq:lse}:
\begin{equation} \label{eq:speclse}
	\hat{f}^{(n)}_{m,k} \in {\arg\min}_{f : \R^d \rightarrow \R \text{ is an } (m,k)\text{-spectrahedral function}} ~~~ \frac{1}{n} \sum_{i=1}^n (y^{(i)} - f(x^{(i)}))^2.
\end{equation}
Here the parameters $m,k$ are specified by the user.

First, we investigate in Section \ref{s:expressive} the expressive power of $(m,k)$-spectrahedral functions.  Our approach to addressing this question is statistical in nature and it proceeds in two steps.  We begin by deriving upper bounds on the error of the constrained estimator \eqref{eq:speclse} (under suitable assumptions on the data $\{(x^{(i)},y^{(i)})\}_{i=1}^n$ supplied to the estimator \eqref{eq:speclse}), which entails computing pseudo-dimension of a set that captures the complexity of the class of spectrahedral functions.  As is standard in statistical learning theory, this error decomposes into an estimation error (due to finite sample size) and an approximation error (due to constraining the estimator \eqref{eq:speclse} to a proper subclass of convex functions).  We then compare these to known minimax lower bounds on the error of any procedure for identifying a convex function \cite{Guntuboyina2012,Balazs2015}.  Combined together, for the case of fixed $k$ (as a function of $m$) we obtain tight lower bounds on how well an $(m,k)$-spectrahedral function can approximate a Lipschitz convex function over a compact convex domain, and on how well a linear image of an $(m,k)$-spectraplex can approximate an arbitrary convex body (see Theorem \ref{t:approx}).  To the best of our knowledge, such bounds have only been obtained previously in the literature for the case $k=1$, e.g., how well $m$-polytopes can approximate arbitrary convex bodies \cite{Bronshtein2008, Dudley74}.

Second, we investigate in Section \ref{s:computational} the performance of an alternating minimization procedure to solve \eqref{eq:speclse} for a user-specified $m,k$.  This method is a natural generalization of a widely-used approach for fitting $m$-polyhedral functions, and it was first described in \cite{Soh19} for the case of positively homogenous convex regression with $k=m$.  We investigate the convergence properties of this algorithm under the following problem setup.  Consider an $(m,k)$-spectrahedral function $f_* : \R^d \rightarrow \R$.  Assuming that the covariates $x^{(i)}, ~ i=1,\dots,n$ are i.i.d. Gaussian and each $y^{(i)} = f_*(x^{(i)}) + \varepsilon_i, ~ i=1,\dots,n$ for i.i.d. Gaussian noise $\varepsilon_i$, we show in Theorem \ref{t:conv_bnd} that the alternating minimization algorithm is locally linearly convergent with high probability given sufficiently large $n$.   A key feature of this analysis is that the requirements on the sample size $n$ and the assumptions on the quality of the initial guess are functions of a `condition number' type quantity associated to $f_*$, which (roughly speaking) measures how $f_*$ changes if the parameters that describe it are perturbed.

Finally, in Section \ref{s:examples} we give empirical evidence of the utility of our estimator \eqref{eq:speclse} on both synthetic datasets as well as data arising from real-world applications.

\subsection{Related Work}\label{sec:related}
There are three broad topics with which our work has a number of connections, and we describe these in detail next.

First, we consider our results in the context of the recent literature in optimization on lift-and-project methods (see the recent survey \cite{Fawzi2020LiftingFS} and the references therein).  This body of work has studied the question of the most compact description of a convex body as a linear image of an affine section of a cone, and has provided lower bounds on the sizes of such descriptions for prominent families of cone programs such as LP, SOCP, and SDP.  This literature has primarily considered exact descriptions, and there is relatively little work on lower bounds for approximate descriptions (with the exception of the case of polyhedral descriptions).  The present paper may be viewed as an approximation-theoretic complement to this body of work, and we obtain tight lower bounds on the expressive power of $(m,k)$-spectrahedral functions (and on linear images of the $(m,k)$-spectraplex) for bounded $k > 1$.

Second, recent results provide algorithmic guarantees for the widely used alternating minimization procedure for fitting $m$-polyhedral functions \cite{Ghoshetal2020Gaussian, Ghoshetal2020SmallBall}; this work gives both a local convergence analysis as well as a dimension reduction strategy to restrict the space over which one needs to consider random initializations.  In comparison, our results provide only a local convergence analysis, although we do so for a more general alternating minimization procedure that is suitable for fitting general $(m,k)$-spectrahedral functions.  We defer the study of a suitable initialization strategy to future work (see Section \ref{sec:discussion}).

Finally, we note that there is prior work on fitting non-polyhedral functions in the convex regression problem.  Specifically, \cite{Hoburgetal2015} suggests various heuristics to fit a log-sum-exp type function, which may be viewed as a `soft-max' function.  However, these methods do not come with any approximation-theoretic or algorithmic guarantees.  In recent work, \cite{Soh19} considered the problem of fitting a convex body given support function evaluations, i.e., the case of positively homogenous convex regression, and proposed reconstructions that are linear images of an $(m,m)$-spectraplex; in this context, \cite{Soh19} provided an asymptotic statistical analysis of the associated estimator and first described an alternating minimization procedure that generalized the $m$-polyhedral case, but with no algorithmic guarantees.  In comparison to \cite{Soh19}, the present paper considers the more general setting of convex regression and also allows for the spectrahedral function to have additional block-diagonal structure, i.e., general $(m,k)$-spectrahedral reconstructions.  Further, we provide algorithmic guarantees in the form of local convergence analysis of the alternating minimization procedure and we provide approximation-theoretic guarantees associated to $(m,k)$-spectrahedral functions (which rely on finite sample rather than asymptotic statistical analysis).

\subsection{Notation}
For $\A = (A_1, \ldots, A_d) \in (\mathbb{S}^m_k)^d$, we define for $x \in \RR^d$ the linear pencil $\A[x]: = \sum_{i=1}^d x_i A_i \in \mathbb{S}^m_k$. The usual vector $\ell_2$ norm is denoted $\|\cdot\|_2$ and the sup norm by $\|\cdot\|_{\infty}$. The matrix Frobenius norm is denoted by $\|\cdot\|_F$, and the matrix operator norm by $\|\cdot\|_{op}$. We denote by $B_d(x, R)$ the ball in $\RR^d$ centered at $x \in \RR^d$ with radius $R > 0$.

\section{Expressiveness of spectrahedral functions via statistical risk bounds}\label{s:expressive}

In this section, we first obtain upper bounds on the risk of the $(m,k)$-spectrahedral estimator in \eqref{eq:speclse} decomposed into the approximation error and estimation error. We then compare this upper bound with known minimax lower bounds on the risk for certain classes of convex functions. This provides lower bounds on the approximation error of $(m,k)$-spectrahderal functions to these functions classes.

\subsection{General Upper Bound on the Risk}

To obtain an upper bound on the risk of the estimator \eqref{eq:speclse}, we use the general bound obtained in \cite[Section 4.1]{HanWellner2016}. To give the statement, consider first the following general framework. Let $(x^{(1)}, y^{(1)}), \ldots, (x^{(n)},y^{(n)})$ be observations satisfying
\begin{align}\label{e:gen_model}
y^{(i)} = f_*(x^{(i)}) + \ee_i,
\end{align}
for a function $f_* : \RR^d \to \RR$ contained in some function class $\mathcal{F}$. We assume the errors $\ee_i$ are i.i.d. mean zero Gaussians with variance $\sigma^2$. Now, let $\{\mathcal{F}_{\ell}\}_{\ell \in \mathbb{N}}$ be a collection of function classes of growing complexity with $m$. For each $m$, define the constrained least squares estimator
\[\hat{f}^{(n)}_m := \mathrm{argmin}_{f \in \mathcal{F}_m} \sum_{i=1}^n (y^{(i)} - f(x^{(i)}))^2.\]
We consider the risk of this estimator in 
the random design setting\footnote{One can also consider the risk in the fixed design setting, where one assumes the covariates $\{x^{(i)}\}_{i=1}^n$ are fixed, and risk bounds proved in \cite{HanWellner2016} include this case. The results in this work can be directly extended to this case as well by applying the corresponding results.}, 
where we assume $x^{(1)}, \ldots, x^{(n)}$ are i.i.d. random vectors in $\RR^d$ with distribution $\mu$. The risk 
is then defined by
\[\|\hat{f}^{(n)}_{\ell} - f_*\|^2_{\mu} := \int_{\RR^d} (\hat{f}^{(n)}_m(x) - f_*(x))^2 \dint \mu(x).\]
Additionally, assume that both $f_*$ and $\mathcal{F}_m$ are uniformly bounded by a positive and finite constant $\Gamma$.

As is standard in the theory of empirical processes, the rate is determined by the complexity of the class $\mathcal{F}_m$, which in this case is determined by the pseudo-dimension of the set
\begin{align}\label{e:HmX}
H_{m} := \{z \in \RR^n: z = (f(x^{(1)}), \dots, f(x^{(n)})) \text{ for some } f \in \mathcal{F}_m\}.
\end{align}
Recall that the pseudo-dimension of subset $B \subset \RR^n$, denoted by Pdim($B$), is defined as the maximum cardinality of a subset $\sigma \subseteq \{1, \ldots, n\}$ for which there exists $h \in \RR^n$ such that for every $\sigma' \subseteq \sigma$, one can find $a \in A$ with $a_i < h_i$ for $i \in \sigma'$ and $a_i > h_i$ for $i \in \sigma \backslash \sigma'$. 

Theorem 4.2 in \cite{HanWellner2016}, stated below, provide an upper bound on the risk of $\hat{f}^{(n)}_m$ 
split into approximation error and estimation error. 

\begin{theorem}\label{t:wellnerriskbound}
Let $n \geq 7$. Suppose there is a constant $D_{m} \geq 1$ such that $\mathrm{Pdim}(H_{m}) \leq D_{m}$. Then, there exists an absolute constants $c$ such that
\begin{align}
    \|\hat{f}^{(n)}_{m} - f_*\|_{\mu}^2 \leq c \left( \inf_{f \in \mathcal{F}_{m}} \|f - f_*\|_{\mu}^2 + \frac{\max\{\sigma^2, \Gamma^2\} D_{m} \log n}{n}\right).
\end{align}
\end{theorem}

The $(m,k)$-spectrahedral estimator \eqref{eq:speclse} is a special case of the estimator $\hat{f}^{(n)}_m$ when $\mathcal{F}$ is the class of convex functions $f:\RR^d\to \RR$ and $\mathcal{F}_{m}$ is the class of $(m,k)$-spectrahedral functions as in Definition \ref{def:specfxn}, denoted by $\mathcal{F}_{m,k}$. 
Since the class is parameterized by $d+1$ matrices in $\mathbb{S}_k^m$, we define for each $m \in \mathbb{N}$ and $k=1, \ldots, m$,
\begin{align}
(\hat{A}_1, \ldots, \hat{A}_d, \hat{B}) \in \mathrm{argmin}_{A_1, \ldots, A_d, B \in \mathbb{S}_k^m} \sum_{j=1}^n \left[y^{(j)} - \lambda_{\max}\left(\sum_{i=1}^d x_i^{(j)}A_i + B\right)\right]^2,
\end{align}
and define the $(m,k)$-spectrahedral estimator of $f_*$ by 
\[\hat{f}_{m,k}(x) := \lambda_{\max}\left(\sum_{i=1}^d x_i\hat{A}_i + \hat{B}\right).\] 

We also define the estimator when $\mathcal{F}$ is the class of support functions of convex bodies (compact and convex subsets) in $\RR^d$, denoted by $\mathcal{K}$, and $\mathcal{F}_m$ is the subclass consisting of positively homogeneous $(m,k)$-spectrahedral functions, or equivalently, support functions of $(m,k)$-spectratopes. This corresponds to the case when the offset matrix $B = 0$.
In this setting, we assume we are given observations $(u^{(1)}, y^{(1)}), \ldots, (u^{(n)}, y^{(n)}) \in \mathbb{S}^{d-1} \times \RR$ satisfying
\[y^{(i)} = h_{K_*}(u^{(i)}) + \ee_i,\]
where 
$h_K(u) := \sup_{x \in K} \langle u,x \rangle$, $u \in \mathbb{S}^{d-1}$, is the support function of a set $K_* \in \mathcal{K}$. 
We denote the class of $(m,k)$-spectratopes, or linear images of $\mathcal{S}_{m,k}$ in $\RR^d$ by $\mathcal{L}(\mathcal{S}_{m,k})$.
To define the $(m,k)$-spectratope estimator, 
let
\[(\hat{A}_1, \ldots, \hat{A}_d) \in \mathrm{argmin}_{A_1, \ldots, A_d,\in \mathbb{S}_k^{m}} \sum_{j=1}^n \left[Y_i - \lambda_{\max}\left(\sum_{i=1}^du^{(j)}_iA_i\right)\right]^2,\]
and define
\[\hat{K}_{m,k} = \{z \in \RR^d : z = (\langle \hat{A}_1, X\rangle, \ldots, \langle \hat{A}_d, X \rangle) \text{ for some } X \in \mathcal{S}_{m,k}\}.\]

For support function estimation, we notate the risk in terms of the convex bodies. 
Letting $\nu$ denote the probability distribution on $\mathbb{S}^{d-1}$ of $u^{(1)}$, we define the risk
\[\ell_{\nu}^2(\hat{K},K) := \int_{\mathbb{S}^{d-1}} (h_{\hat{K}}(u) - h_{K}(u))^2 \dint \nu(u).\]

In the following lemma, we prove an upper bound on the pseudo-dimension of the relevant set \eqref{e:HmX} needed to apply Theorem \ref{t:wellnerriskbound} for the estimators $\hat{f}_{m,k}$ and $\hat{K}_{m,k}$.  

\begin{lemma}\label{l:PdimBound}
For $m, k \in \mathbb{N}$ such that $k$ divides $m$, define for $x^{(1)}, \ldots, x^{(d)} \in \RR^d$,
\begin{align*}
H_{m,k} &:= \bigg\{ z = \left(\lambda_{\max}\left(\mathcal{A}[x^{(1)}] + B\right), \ldots, \lambda_{\max}\left(\mathcal{A}[x^{(n)}]  + B\right)\right) \in \RR^n \\
& \qquad \quad \text{ for some } \mathcal{A} \in (\mathbb{S}^m_k)^d, B \in \mathbb{S}_k^{m}\bigg\},
\end{align*}
and for $u^{(1)}, \ldots, u^{(d)}  \in \mathbb{S}^{d-1}$,
\begin{align*}
\tilde{H}_{m,k} &:= \bigg\{z = \left(\lambda_{\max}\left(\mathcal{A}[u^{(1)}]\right), \ldots, \lambda_{\max}\left(\mathcal{A}[u^{(n)}]\right)\right) \in \RR^n \text{ for some } \mathcal{A} \in (\mathbb{S}^m_k)^d\bigg\},
\end{align*}
Then, there exists absolute constant $c_1, c_2 > 0$ such that
\begin{align*}
\text{Pdim}(H_{m,k}) \leq c_1 km(d+1) \log(c_2 n/k)
\quad
\text{and} 
\quad
\text{Pdim}(\tilde{H}_{m,k}) \leq c_1 km d \log(c_2 n/k).
\end{align*}
\end{lemma}

To prove the lemma, we need the following known result (see for instance, Lemma 2.1 in \cite{BartlettMaiorovMeir1998}):

\begin{proposition}\label{p:distinctsigns}
Let $p_1, \ldots, p_n$ be fixed polynomials of degree at most $m$ in $D$ variables for $D \leq m$. The number of distinct sign vectors $(\text{sgn}(p_1(A), \ldots, \text{sgn}(p_n(A)))$ that can be obtained by varying $A \in \RR^D$ is at most $2\left(\frac{2e n m}{D}\right)^D$.
\end{proposition}

\begin{proof}(of Lemma \ref{l:PdimBound})
Assume that the pseudo-dimension of $H_{m,k} \subset \RR^n$ 
is $\rho$. By the definition of pseudo-dimension, the size of the collection of sign vectors
\[\mathcal{G}_{m,k} := \{(\text{sgn}(\lambda_{\max}(\mathcal{A}[x^{(1)}] + B), \ldots, \text{sgn}(\lambda_{\max}(\mathcal{A}[x^{(n)}] + B))) : \mathcal{A} \in (\mathbb{S}_k^{m})^{d}, B \in \mathbb{S}_k^m\}\]
must be at most $2^{\rho}$. For each $i$, \[\text{sgn}(\mathcal{A}[x^{(i)}] + B) = \text{sgn}(\min\{p_1(\mathcal{A},B;x^{(i)}), \ldots, p_m(\mathcal{A}, B;x^{(i)})\}),\]
where $p_{\ell}(\A, b; x^{(i)}) = \det(-(\mathcal{A}[x^{(i)}] + B)_{\ell:\ell})$ is the determinant of the $\ell \times \ell$ principal submatrix of $-\mathcal{A}[x^{(i)}]- B$. Indeed, $\lambda_{\max}(\mathcal{A}[x^{(i)}]+ B) \leq 0$ 
if and only if all of these determinants are non-negative. Thus, the size of $\mathcal{G}_{m,k}$ is the same size of
\[\mathcal{I}_{m,k} := \{ (\text{sgn}(p(\mathcal{A}, B; x^{(1)})), \ldots, \text{sgn}(p(\mathcal{A}, B;x^{(n)}))) : \mathcal{A} \in (\mathbb{S}_k^{m})^{d+1}\}),\]
where for each $i$, $p(\mathcal{A},B;x^{(i)}) := \min\{p_1(\mathcal{A}, B;x^{(i)}), \ldots, p_m(\mathcal{A}, B; x^{(i)})\}$ is a piecewise polynomial in $\mathcal{A}$. To bound the size of $\mathcal{I}_{m,k}$, we use the idea from \cite{BartlettMaiorovMeir1998}. We can partition $(\mathbb{S}_k^m)^{d+1}$ into at most $mn$ regions over which the vector is coordinate-wise a fixed polynomial. Then we apply Proposition \ref{p:distinctsigns}.

We have $n$ polynomials of degree at most $m$ in up to $D = (d+1) k m$ variables, i.e. the number of degrees of freedom of $d+1$ $m \times m$ $k$-block matrices. 
Thus, the number of distinct sign vectors in $\mathcal{I}_m$ satisfies $|\mathcal{I}_m| \leq 2mn \left(\frac{2en}{(d+1)k}\right)^{(d+1)km}$. This implies that $2^{\rho} \leq 2mn \left(\frac{2en}{(d+1)k}\right)^{(d+1)km}$, and hence
\begin{align*}
\rho \leq \frac{(d+1)km}{\log 2} \log\left(\frac{2en}{(d+1)k}\right) + \frac{\log(2mn)}{\log 2} \leq c_1 km(d+1) \log\left(\frac{c_2 n}{k}\right).
\end{align*}
The second claim follows similarly, where instead $D = dkm$.
\end{proof}

We can now obtain an upper bound on the risk of the estimators $\hat{f}_{m,k}$ and $\hat{K}_{m,k}$. Recall that 
we assume $f_*$ and functions in $\mathcal{F}_{m,k}$ are uniformly bounded by some $\Gamma \in (0,\infty)$, and for support function estimation we assume $K_*$ and elements of $\mathcal{L}(\mathcal{S}_{m,k})$ are contained in $B_d(0, \Gamma)$.

\begin{theorem}\label{t:CRrate}
\begin{itemize}
    \item[(i)] For any convex function $f_*: \RR^d \to \RR$, there exist absolute constants $c$ and $b$ such that
\[\|\hat{f}_{m,k} - f_*\|_{\mu}^2 \leq c \left(\inf_{f \in \mathcal{F}_{m,k}}\|f - f_*\|^2_{\mu} + \max\{\sigma^2, \Gamma^2\}k m(d+1) \frac{\log(b n/k)}{n}\right).\]
\item[(ii)]For any convex body $K_*$ in $\RR^d$,
\[\ell_{\nu}^2(\hat{K}_{m,k},K_*) 
\leq c \left(\inf_{S \in \mathcal{L}(\mathcal{S}_{m,k})} \ell^2_{\nu}(S,K_*) + \max\{\sigma^2, \Gamma^2\}\frac{mk}{n}d\log(b n)\right)\]
\end{itemize}
\end{theorem}

\begin{proof}
This result follows from Theorem \ref{t:wellnerriskbound} and Lemma \ref{l:PdimBound}.
\end{proof}

\subsection{Minimax Rates} 

The minimax risk for estimating a function in the class $\mathcal{F}$ from $\{x^{(i)}, y^{(i)}\}_{i=1}^n$ in the random design setting is defined by
\[R_{\mu}(n, \mathcal{F}) := \min_{\hat{f}} \max_{f \in \mathcal{F}}  \|\hat{f} - f\|_{\mu}.\]

In Table \ref{table:minimax} we summarize known rates as $n \to \infty$ of this minimax risk for certain sub-classes of convex functions. 
First consider the class $\mathcal{F}_{m,k}(\Omega)$ of functions in $\mathcal{F}_{m,k}$ with compact and convex domain $\Omega \subset \RR^d$. 
In this case, the approximation error in the risk bound is zero and the rate of convergence is $O\left(\frac{\log n}{n}\right)$. This is the best rate we can achieve when the domain $\Omega$ satisfies a certain smoothness assumption (see \cite[Theorem 2.6]{HanWellner2016}) 
and also appealing to the fact that $\mathcal{F}_{m,1} \subseteq \mathcal{F}_{m,k}$. 
Otherwise, the best lower bound we have is $O\left(\frac{1}{n}\right)$ using standard arguments for parametric estimation. 

Additionally we consider two non-parametric sub-classes of convex functions. First is Lipschitz convex regression, where we assume the true function $f_*$ belongs to the class $\mathcal{C}_L(\Omega)$ of $L$-Lipschitz convex functions with convex and compact full-dimensional support $\Omega \subset \RR^d$.
Second is support function estimation, where we assume the true function is the support function of a set $K$ belonging to the collection $\mathcal{K}(\Gamma)$ of convex and compact subsets of $\RR^d$ contained in the ball $B_d(0, \Gamma)$ for some finite $\Gamma > 0$. In both settings, the usual LSE over the whole class is minimax sub-optimal \cite{KurConvexRegression, KurSupport}, necessitating a regularized LSE to obtain the minimax rate.


\begin{table}[h!]
  \caption{Minimax Rates for Sub-Classes of Convex Functions}
  \centering 
    \begin{tabular}{c || c | c | c | c}
    $\mathcal{F}$  & $\mathcal{F}_{m,k}(\Omega)$, for $\Omega$ smooth \cite{HanWellner2016} & $\mathcal{C}_L(\Omega)$  \cite{Balazs2015}  
    &  $\mathcal{K}(\Gamma)$  \cite{Guntuboyina2012}\\
     \midrule 
 $R_{\mu}(n, \mathcal{F})$  \,  &  $\frac{\log n}{n}$ 
 & $n^{-\frac{4}{d+4}}$ 
 & $n^{-\frac{4}{d+3}}$
 
    \end{tabular}
\label{table:minimax}
  \end{table}

\subsection{Approximation Rates}

For Lipschitz convex regression, Lemma 4.1 in \cite{Balazs2015} implies the following: for $f_* \in \mathcal{C}_L(\Omega)$,
\begin{align}\label{e:CRapprox}
\inf_{f \in \mathcal{F}_{m,1}}\|f - f_*\|_{\mu} \leq  \inf_{f \in \mathcal{F}_{m,1}}\|f - f_*\|_{\infty} \leq c_{d, \Omega, L} m^{-2/d}.
\end{align}
For support function estimation, let $d_H(S,K) := \|h_S - h_K\|_{\infty}$ denote the Hausdorff distance between any $S$ and $K$ in $\mathcal{K}$. A classical result of Bronshtein (see Section 4.1 in \cite{Bronshtein2008}) implies that 
\begin{align}\label{e:SFapprox}
  \inf_{S \in \mathcal{L}(\mathcal{S}_{m,1})} \ell_{\nu}(S,K) \leq \inf_{\mathcal{L}(\mathcal{S}_{m,1})} d_H(S,K) \leq c_{d,\Gamma}m^{-2/(d-1)},  
\end{align}
This result is also the core of the proof of \eqref{e:CRapprox}.

We first show that inserting \eqref{e:CRapprox} and \eqref{e:SFapprox} into Theorem \ref{t:CRrate} and optimizing over $m$ gives general upper bounds on the risk for our $(m,k)$-spectrahedral estimators. These rates match the minimax rate up to logarithmic factors for fixed $k > 0$, 
and even when $k$ is allowed to depend logarithmically on $m$. 


\begin{corollary}\label{t:rates}
Suppose $k_m = f(m)$ for a non-decreasing and differentiable function $f: \RR \to (0, m]$. 
\begin{itemize}
    \item[(a)] (Lipschitz convex regression) Suppose $f_* \in \mathcal{C}_L(\Omega)$ and define the function
\[g(m) := f'(m) m^{\frac{2d+ 4}{d}} + f(m) m^{\frac{d + 4}{d}}.\]
Then, 
for $\alpha_n = g^{-1}\left(\frac{2n}{d(d+1) \max\{\sigma^2,\Gamma^2\} \log(bn)}\right)$,
\begin{align}\label{e:rate1}
\inf_{m \geq 1} \|\hat{f}_{m,k_m} - f_*\|^2_{\mu} \leq c_{d, \Omega, \Gamma}\left(\alpha_n^{-4/d} + \max\{\sigma^2,\Gamma^2\}(d+1) \alpha_n f(\alpha_n) \frac{\log(bn)}{n}\right),
\end{align}
\item[(b)] (Support function estimation) Suppose $K_* \in \mathcal{K}(\Gamma)$ and define the function
\[g(m) := f'(m) m^{\frac{2(d+1)}{d-1}} + f(m) m^{\frac{d+3}{d-1}}.\]
Then, 
for $\alpha_n = g^{-1}\left(\frac{2n}{(d-1)d\max\{\sigma^2,\Gamma^2\} \log(bn)}\right)$,
\begin{align}
\inf_{m \geq 1} \ell^2_{\nu}(\hat{K}_{m,k_m},K_*) \leq c_{d, \Gamma}\left(\alpha_n^{-\frac{4}{d-1}} + \max\{\sigma^2,\Gamma^2\}(d+1) \alpha_n f(\alpha_n) \frac{\log(bn)}{n}\right).
\end{align}
\end{itemize}
\end{corollary}

We now provide two specific examples for particular functions $f$:
\begin{itemize}
\item[(i)] If $f(m) = km^r$ for fixed $k > 0$ and $r \in [0,1]$, then
\[\inf_{m \geq 1} \|\hat{f}_{m,k_m} - f_*\|_{\mu}^2 \leq O\left(n^{-\frac{4}{(r+1)d + 4}}\log(bn)^{\frac{4}{(r+1)d + 4}}\right),\]
and
\[\inf_{m \geq 1} \ell_{\nu}^2(\hat{K}_{m,k_m},K_*) \leq O\left( n^{-\frac{4}{(r+1)(d-1) + 4}}\log(bn)^{\frac{4}{(r+1)(d-1)+ 4}}\right).\]
\item[(ii)] If $f(m) = \log m$, then $\alpha_n = O\left(n^{\frac{d}{d+4}}\log(n)^{-\frac{2d}{d+4}}\right)$, and
\[\inf_{m \geq 1} \|\hat{f}_{m,k_m} - f_*\|_{\mu}^2 \leq O\left( n^{-\frac{4}{d+4}}\log(n)^{\frac{8}{d+4}}\right),\]
and
\[\inf_{m \geq 1} \ell_{\nu}^2(\hat{K}_{m,k_m},K_*) \leq O\left(n^{-\frac{4}{d+3}}\log(n)^{\frac{8}{d+3}}\right).\]
Indeedn, the inverse of $h(x) := x^a \log(x)$ is $h^{-1}(x) = \left(\frac{ax}{W(ax)}\right)^{1/a}$, where $W$ is the Lambert W function. The bound then follows from the fact that $W$ satisfies $\log W(x) = \log x - W(x)$ and as $x \to \infty$, $W(x) \sim \log(x)$.
\end{itemize}
 
\begin{remark}
For the case $k=1$, Corollary \ref{t:rates} recovers the results in \cite{Guntuboyina2012} and \cite{HanWellner2016} showing these estimators obtain the minimax rate (up to logarithmic factors) for the relevant class of functions. 
\end{remark}

\begin{proof}
We prove equation \eqref{e:rate1}, and the second statement follows by a similar argument. By Theorem \ref{t:rates} and \eqref{e:CRapprox},  
\[ \|\hat{f}_{m,k_m} - f_*\|^2_{\mu} \leq c_{d, \Omega, L} \left(m^{-4/d} + \frac{\max\{\sigma^2, \Gamma^2\} (d+1) f(m)m}{n}\log(b n)\right).\]
The $m_{\star}$ that minimizes the expression in the parentheses above satisfies
\begin{align*}
    0 &= -\frac{4}{d}(m_{\star})^{-\frac{4}{d}-1} + \frac{\max\{\sigma^2, \Gamma^2\} (d+1)\log(bn)}{n}\left(f'(m_{\star})m_{\star} + f(m_{\star})\right),
\end{align*}
or equivalently,
\begin{align*}
    \frac{4n}{d(d+1)\max\{\sigma^2, \Gamma^2\}\log(bn)} &= f'(m_{\star})m_{\star}^{\frac{2d + 4}{d}} + f(m_{\star})m_{\star}^{\frac{d+4}{d}} = g(m_{\star}).
\end{align*}
Then, $m_{\star} = g^{-1}\left(\frac{4n}{d(d+1)\max\{\sigma^2, \Gamma^2\}\log(bn)}\right)$ and plugging back into the upper bound gives the result.
\end{proof}

As stated previously, an important observation from Corollary \ref{t:rates} is that when $k_m = k$ is a fixed constant that does not depend on $m$, the risk bounds for an optimal choice $m_{\star}$ match (up to logarithmic factors) the minimax lower bounds of the classes $\mathcal{C}_L(\Omega)$ and $\mathcal{C}(\Gamma)$. This indicates that the approximation error for the classes $\mathcal{F}_{m,k}$ and $\mathcal{S}_{m,k}$ for fixed $k$ cannot be improved from what was used in the proof. Indeed, this statistical risk analysis provides the following main result of this section: approximation rate lower bounds for the parametric classes $\mathcal{F}_{m,k}$ and $\mathcal{S}_{m,k}$. 

\begin{theorem}\label{t:approx}
Suppose there exists an absolute constant $c > 0$ and $t \in [0,1]$ such that $k_m \leq c m^{t}$ for all $m$ large enough. Let $f_* \in \mathcal{C}_L(\Omega)$. 
For all $\ee > 0$, for all $m$ large enough,
\[ \inf_{f \in \mathcal{F}_{m,k_m}} \|f - f_*\|_{\infty} \geq c_{d, L, \Omega} m^{-2(1 + t)/d - \ee}.\]
Also, let $K_* \in \mathcal{K}(\Gamma)$. For all $\ee > 0$, for all $m$ large enough,
\[\inf_{S \in \mathcal{S}_{m,k_m}} d_H(S,K_*) \geq c_{d, \Gamma} m^{-2(1 + t)/(d-1) - \ee}.\] 
\end{theorem}

\begin{remark}
For constant $k$ (i.e. $t = 0$), Theorem \ref{t:approx} implies 
\[\inf_{f \in \mathcal{F}_{m,k}} \|f - f_*\|_{\infty} = \tilde{O}(n^{-2/d})\,\text{  and  } \,\inf_{S \in \mathcal{S}_{m,k}} d_H(S,K_*) = \tilde{O}(n^{-2/(d-1)}),\]
where the $\tilde{O}$ notation ignores polylogarithmic factors.
\end{remark}

\begin{proof}
We argue by contradiction. Suppose that for all $m > 0$,
\[\inf_{f \in \mathcal{F}_{m,k}} \|f - f_*\|^2_{\mu} \leq c_1 m^{-r},\]
for some constant $c_1$ (that may depend on $L$ and $\Omega$) and fixed $r > \frac{4}{d}(1 + t)$. Then by Theorem \ref{t:CRrate}, there exist constants $c_2$, $b$ such that
\begin{align*}
n^{-4/(d+4)} \leq c_2 \inf_{m > 0} \left( m^{-r} + \max\{\sigma^2, \Gamma^2\} m^{t+1} (d+1) \frac{\log(b n/k)}{n}\right).
\end{align*}
The infimum on the right side is achieved at $m_{\star} = \left(\frac{rn}{\max\{\sigma^2, \Gamma^2\} k (d+1) \log(bn/k)}\right)^{\frac{1}{t + r+1}}$, and thus
\begin{align*}
n^{-4/(d+4)} 
&\leq c_2 n^{-\frac{r}{t + r+1}}\log(bn/k)^{\frac{r}{t + r + 1}}(\max\{\sigma^2, \Gamma^2\} k (d+1))^{\frac{r}{t + r + 1}}\left[r^{\frac{-r}{t + r+1}} + r^{\frac{t + 1}{t + r+1}} \right].
\end{align*}
For this inequality to hold for all $n$, it must be that $r \leq \frac{4}{d}(1 + t)$, a contradiction. 
The second statement is proved similarly.
\end{proof}


\section{Computational Guarantees}\label{s:computational}


\subsection{Alternating Minimization Algorithm}

We now describe an alternating minimization algorithm to solve the non-convex optimization problem \eqref{eq:speclse}. 
Let $\xi_i = (x^{(i)},1) \in \RR^{d+1}$ for each $i=1, \ldots, n$ and let $\A_{*} \in (\mathbb{S}_k^m)^{d+1}$ be the true underlying parameters. That is, for each $i=1, \ldots, n$, we observe 
\[y_i = \lambda_{\max}(\A_*[\xi^{(i)}]) + \ee_i.\]
We assume the $\ee_i$'s are i.i.d. mean zero Gaussian noise with variance $\sigma^2$. 

One step of the algorithm starts with a fixed parameter $\A \in (\mathbb{S}_k^m)^{d+1}$. Then, compute the maximizing eigenvector $u^{(i)} \in \mathbb{S}^{m-1}$, $i=1, \ldots, n$, such that for $U^{(i)} = u^{(i)}(u^{(i)})^T$, $\langle U^{(i)}, \A[\xi^{(i)}] \rangle = \lambda_{\max}\left(\A[\xi^{(i)}]\right)$. With the $U^{(i)}$'s fixed, update $\A$ by solving the linear least squares problem:
\begin{align}\label{e:lls_Aupdate}
\A^{+} \in \mathrm{argmin}_{\A \in (\mathbb{S}_k^m)^{d+1}} \frac{1}{n} \sum_{i=1}^n \left(y^{(i)} - \langle U^{(i)}, \A[\xi^{(i)}] \rangle\right)^2, 
\end{align}
where $\langle U^{(i)}, \A[\xi^{(i)}] \rangle = \langle \A, \xi^{(i)} \otimes U^{(i)} \rangle = \sum_{j=1}^d \langle A_j, \xi^{(i)}_jU^{(i)}\rangle$.

\begin{algorithm}[ht]
    \caption{Alternating Minimization for Spectrahedral Regression}
    \textbf{Input}: Collection of inputs and outputs $\{(x^{(i)}, y^{(i)})\}_{i=1}^n$; initialization $\A \in (\mathbb{S}_k^{m})^{d+1}$ \\
    \textbf{Algorithm}: Repeat until convergence
    
   \hspace{.2in} \textbf{Step 1}: Update optimal eigenvector $u^{(i)} \leftarrow \lambda_{\max}(\A[\xi^{(i)}])$
    
    \hspace{.2in} \textbf{Step 2}: Update $\A$ by solving \eqref{e:lls_Aupdate}. $\A^{+} \leftarrow (\Xi_{\mathcal{A}}^T\Xi_{\mathcal{A}})^{-1}\Xi_{\mathcal{A}}^Ty$, where $\Xi_{\mathcal{A}}^T = (\xi^{(1)} \otimes U^{(1)} | \cdots | \xi^{(n)} \otimes U^{(n)}) \in \RR^{(d+1)m^2 \times n}$.

    \textbf{Output}: Final iterate $\A$
\end{algorithm}

\subsection{Convergence Guarantee}
The following result shows that under certain conditions, this alternating minimization procedure converges geometrically to a small ball around the true parameters given a good initialization. To state the initialization condition in the result, we define for $\A \in (\mathbb{S}_k^{m})^d$ the similarity transformation $\mathcal{O}(\A) = (OA_1O^T, \ldots, OA_dO^T)$ for an orthogonal $m \times m$ matrix $O$. Note that the eigenvalues of $\A[x]$ for $x \in \RR^d$ are invariant under any $\mathcal{O}$. In the following we only consider the setting where $k = m$, and denote $\mathbb{S}^{m} := \mathbb{S}_m^{m}$.

The proof of the following result appears after the statement and it depends on multiple lemmas that we state and prove in the appendix.

\begin{theorem}\label{t:conv_bnd}
Assume $X$ is a standard Gaussian random vector in $\RR^d$ 
and let $\xi = (X,1) \in \RR^{d+1}$.
Also suppose that the true parameter $\A_* \in (\mathbb{S}^m)^{d+1}$ satisfies the following spectral condition: 
\begin{align}\label{e:eigengap}
\inf_{u \in \mathbb{S}^{d-1}} \lambda_1(\A_*[u]) - \lambda_2(\A_*[u]) := \kappa > 0,
\end{align}
where $\lambda_1 := \lambda_{\max}$ and $\lambda_2$ is the second largest eigenvalue. Let $\tau \in (0,1)$. There exist constants $c_i$, $i=1, \ldots, 5$ such that if the initial parameter choice $\A^{(0)}$ satisfies
\[\|\A^{(0)} - \mathcal{O}(\A_*)\|^2_F \leq  \frac{\kappa^2}{16(d+1)m^2}\left(\frac{1 - \tau}{1 + \tau}\right),\]
for some similarity transformation $\mathcal{O}$ and 
\[n \geq c_1(d+1) \max\left\{\tau^{-2}m^{10},\frac{1}{\kappa^2}\left(\frac{1 + \tau}{1 - \tau}\right)\frac{ m^6(d+1)\sigma^2\log(n)^2}{(1-\tau)}\right\},\] 
then the error at iteration $t$ satisfies
\[ \|\A^{(t)} - \mathcal{O}(\A_*)\|_F^2 \leq \left(\frac{3}{4}\right)^{t}\|\A^{(0)} - \mathcal{O}(\A_*)\|_F^2 + \frac{c_3 m^4(d+1)\sigma^2\log(n)^2}{n(1 - \tau)},\] 
with probability greater than $1 - 4e^{-c_4\tau^2 n/m^{10}} - n^{-c_5m^2(d+1)}$.
\end{theorem}



\begin{remark}
The assumption \eqref{e:eigengap} is not always satisfied, but we provide some examples where it is. 
First, consider the case where $d=2$, $A_1$ and $A_2$ are non-commutative matrices, and $A_3 = 0$. Let $a_{ij} = (a^{(1)}_{ij}, a^{(2)}_{ij}) \in \RR^2$. In this case the eigengap is 
\begin{align*}
  \lambda_1(u_1A_1 + u_2A_2) - \lambda_2(u_1A_1 + u_2A_2) 
  &= \langle u, a_{11} - a_{22} \rangle^2 + 4 \langle u, a_{12}\rangle^2.
\end{align*}
This follows from the computation of eigenvalues using the characteristic polynomial and the quadratic formula.
We see that for any $A_1$ and $A_2$ such that $a_{12} \notin \{0, a_{11} - a_{22}\}$, the eigengap is a strictly positive number. For example, if $a_{11} - a_{22} = e_1$ and $2a_{12} = e_2$, then 
\begin{align*}
  \lambda_1(u_1A_1 + u_2A_2) - \lambda_2(u_1A_1 + u_2A_2)
  &=  \langle u, e_1 \rangle^2 + \langle u, e_2 \rangle^2 =  u_1^2 + u_2^2 = 1.
\end{align*}
Another set of parameters $\A_*$ that satisfy condition \eqref{e:eigengap} is when $\lambda_{\max}(\A_{*}[x]) = \|x\|_2$. In fact, for any spectrahedral function $f(x) = \lambda_{\max}(\A_{*}[x])$ that is differentiable for all $x$, $\A_*$ must necessarily satisfy \eqref{e:eigengap}.
\end{remark}

\begin{proof}
First, given assumption \eqref{e:eigengap}, we show that for $n$ large enough, for all parameters $\A$ satisfying for some similarity transform $\mathcal{O}$,
\begin{align}\label{e:initbnd}
\|\A - \mathcal{O}(\A_*)\|^2_F \leq  
\frac{\kappa^2}{16(d+1)m^2}\left(\frac{1 - \tau}{1 + \tau}\right),
\end{align}
the parameter $\A^{+}$ obtained after applying one iteration of the algorithm satisfies
\begin{align}\label{e:finalbnd}
\|\A^{+} - \mathcal{O}(\A_*)\|_F^2 \leq \frac{3}{4}\|\A - \mathcal{O}(\A^{*})\|_F^2 +  O\left(\frac{\log n}{n}\right)
\end{align}
with high probability.

Let $U^{(i)} = u^{(i)}(u^{(i)})^T$ be such that $\lambda_{\max}(\A[\xi^{(i)}]) = \langle U^{(i)}, \A[\xi^{(i)}]\rangle$. The update $\A^{+}$ then equals
\[\A^{+} = (\Xi_{\mathcal{A}}^T\Xi_{\mathcal{A}})^{-1}\Xi_{\mathcal{A}}^Ty,\]
where $\Xi_{\mathcal{A}}^T = (\xi^{(1)} \otimes U^{(1)} | \cdots | \xi^{(n)} \otimes U^{(n)}) \in \RR^{(d+1)m \times m \times n}$.  Note that $(\Xi_{\mathcal{A}} \A)_i = \langle U^{(i)}, \A[\xi^{(i)}]\rangle $. Throughout the rest of the proof, we sometime abuse notation and consider the Kronecker product $\xi \otimes U$ for $\xi \in \RR^{d+1}$ and $U \in \RR^{m \times m}$ to be the vector $\mathrm{Vec}(\xi \otimes U) \in \RR^{(d+1)m^2}$. 

By the invariance $\lambda_{\max}(\A[x]) = \lambda_{\max}(\mathcal{O}(\A)[x])$ for all $\mathcal{O}$, without loss of generality we can assume in the following that $\A_* = \mathcal{O}(\A_*)$ for the transformation $\mathcal{O}$ satisfying assumption \eqref{e:initbnd}. Let $y_{*} \in \RR^{n}$ and $u^{(i)}_* \in \mathbb{S}^{d-1}$ be such that for $U_*^{(i)} := u^{(i)}_*\left(u^{(i)}_*\right)^T$, 
\[y^{*}_i = \langle U^{(i)}_{*}, \A_{*}[\xi^{(i)}]\rangle = \lambda_{\max}(\A_{*}[\xi^{(i)}]).\] 
Also denote by $P_{\Xi_{\mathcal{A}}} = \Xi_{\mathcal{A}}(\Xi_{\mathcal{A}}^T\Xi_{\mathcal{A}})^{-1}\Xi_{\mathcal{A}}^T$ the orthogonal projection matrix onto the span of the columns of $\Xi_{\A}$. Then, we have the following deterministic upper bound:
\begin{align*}
    \|\Xi_{\mathcal{A}}(\A^{+} - \A_{*})\|^2 &= \|P_{\Xi_{\mathcal{A}}} y - \Xi_{\mathcal{A}} \A_{*} \|^2 = \|P_{\Xi_{\mathcal{A}}} y^* + P_{\Xi_{\mathcal{A}}} \ee - \Xi_{\mathcal{A}} \A_{*} \|^2 \\
    & \leq 2\|P_{\Xi_{\mathcal{A}}}(y^{*} - \Xi_{\mathcal{A}} \A_{*})\|^2 + 2\|P_{\Xi_{\mathcal{A}}}\ee\|^2 \\
    & \leq 2 \sum_{i=1}^n \left(\langle U_{*}^{(i)}, \A_{*}[\xi^{(i)}]\rangle - \langle U^{(i)}, \A_{*}[\xi^{(i)}]\rangle \right)^2 + 2\|P_{\Xi_{\mathcal{A}}}\ee\|^2.
\end{align*}
Now, since $\langle U^{(i)} - U^{(i)}_{*}, \A[\xi^{(i)}]\rangle \geq 0$,
\begin{align*}
\left(\langle U^{(i)}_{*}, \A_{*}[\xi^{(i)}]\rangle - \langle U^{(i)}, \A_{*}[\xi^{(i)}]\rangle \right)^2 
&\leq 
\left(\langle U^{(i)}_{*} -  U^{(i)}, \A_{*}'\xi^{(i)}\rangle + \langle U^{(i)} - U_{*}^{(i)}, \A[\xi^{(i)}]\rangle \right)^2 \\
&= 
\left\langle \A - \A_{*}, \xi^{(i)} \otimes (U^{(i)} - U^{(i)}_{*}) \right\rangle^2.
\end{align*}
We also have the lower bound
$\|\Xi_{\mathcal{A}}(\A^+ - \A_*)\|^2 \geq \lambda_{\min}(\Xi_{\mathcal{A}}^T\Xi_U)\|\A^{+} - \A_{*}\|^2$. Thus,
\begin{align}\label{e:Adiff_bnd}
\|\A^{+} - \A_{*}\|^2 &\leq \frac{2}{\lambda_{\min}(\Xi_{\A}^T\Xi_{\A})}\left[ \|\Xi_{\A- \A_*}(\A - \A_{*})\|_2^2 + \|P_{\Xi_{\A}}\ee\|^2\right]\nonumber \\
&\leq \frac{2}{\lambda_{\min}(\Xi_{\A}^T\Xi_{\A})}\left[ \lambda_{\max}(\Xi_{\A- \A_*}^T\Xi_{\A- \A_*}) \|\A - \A_{*}\|^2 + \|P_{\Xi_{\A}}\ee\|^2\right].
\end{align}
where $\Xi_{\A-\A_*} = \left(\xi^{(1)} \otimes (U^{(1)} - U^{(1)}_{*}) | \cdots | \xi^{(n)} \otimes (U^{(n)} - U^{(n)}_{*})\right)$.

Lemmas \ref{l:spectrum_bnd} and \ref{l:unif_cov} then imply the following. For $\tau \in(0,1)$, there exist absolute constants $c_1, c_2$ such that if $n \geq c_1 \tau^{-2}(d+1)m^{10}$, then with probability greater than $1 - 2e^{-c_2\tau^2n/m^{10}}$,
\begin{align}\label{e:lmax_up1}
 \lambda_{\max}(\Xi_{\A- \A_*}^T\Xi_{\A- \A_*}) \leq n \lambda_{\max}(\EE[(\xi \otimes (U- U_{*}))(\xi \otimes (U- U_{*}))^T])\left(1 + \tau\right)
 \end{align}
 for all $\A$ satisfying assumption \eqref{e:initbnd}.
Since $\lambda_{\max}$ is convex function,  Jensen's inequality implies
 \begin{align*}
\lambda_{\max}(\EE[(\xi \otimes (U- U_*))(\xi \otimes (U - U_*))^T]) \leq   \EE[\|\xi \otimes (U - U_*)\|^2].
\end{align*}
Then, by the definition of the Kronecker product,
\[\|\xi \otimes (U - U_*)\|^2 = \sum_{i=1}^d \sum_{j,k=1}^m \xi_i^2 (U - U_*)_{jk}^2 = \|\xi\|_2^2\|U - U_*\|_F^2.\]
Next note that $\|U - U_*\|_F^2 \leq 2 \|u - u_*\|_2^2$, where $U = u u^T$, $U_* = u_* u_*^T$ and $u, u_* \in \mathbb{S}^{d-1}$. Then, by a variation of the  Davis-Kahan Theorem (Corollary 3 in \cite{WangSamworth2015}),
\begin{align*}
    \|u-u_{*}\|_2 &\leq \frac{2^{3/2}\|(\A - \A_*)'\xi\|_{op}}{\lambda_1(\A_*'\xi) - \lambda_2(\A_*'\xi)} 
    \leq 2^{3/2}\kappa^{-1}\|\A - \A_*\|_F.
\end{align*}
Putting the bounds together and using assumption \eqref{e:initbnd}, 
\begin{align}\label{e:lmax_up2}
    \lambda_{\max}(\EE[(\xi \otimes (U- U_*))(\xi \otimes (U - U_*))^T]) &\leq  2^{5/2}\kappa^{-2} \|\A - \A_*\|_F^2 \EE[\|\xi\|_2^2] \nonumber \\
    &\leq 6\kappa^{-2}(d+1) \|\A - \A_*\|_F^2 \leq \frac{3}{8m^2}\left(\frac{1 - \tau}{1 + \tau}\right). 
\end{align}
Plugging the bound \eqref{e:lmax_up2} into \eqref{e:lmax_up1} gives
\begin{align}\label{e:lmax_up_final}
    \lambda_{\max}(\Xi_{\A- \A_*}^T\Xi_{\A- \A_*}) \leq \frac{3n}{8m^2}\left(1 - \tau\right).
\end{align}

Also by Lemmas \ref{l:spectrum_bnd} and \ref{l:unif_cov} if $n \geq c_1 \tau^{-2} (d+1)m^{10}$, then with probability greater that $1 - 2e^{-c_2\tau^2n/m^{10}}$,
\begin{align}\label{e:lmax_low1}
\lambda_{\min}(\Xi_{\A}^T\Xi_{\A}) \geq n\lambda_{\max}(\EE[(\xi \otimes U)(\xi \otimes U)^T]) \left(1 - \tau\right)
\end{align}
for all $\A$ satisfying \eqref{e:initbnd}.
We then have the following lower bound: 
\begin{align}\label{e:lmax_low2}
&\lambda_{\max}(\EE[(\xi \otimes U)(\xi \otimes U)^T]) \geq \frac{1}{(d+1)m^2}\mathrm{Tr}\left[\EE(\xi \otimes U)(\xi \otimes U)^T\right] \nonumber \\
&\quad = \frac{1}{(d+1)m^2}\sum_{i=1}^{d+1} \sum_{j,k=1}^m \EE[\xi_i^2 (u_j u_k)^2] = \frac{1}{(d+1)m^2}\sum_{i=1}^{d+1} \EE[\xi_i^2] = m^{-2}.   
\end{align}
Plugging the bound \eqref{e:lmax_low2} into \eqref{e:lmax_low1} gives
\begin{align}\label{e:lmax_low_final}
  \lambda_{\min}(\Xi_{\A}^T\Xi_{\A})\geq n m^{-2} \left(1 - \tau\right)  ,
\end{align}
and finally combining \eqref{e:lmax_up_final} and \eqref{e:lmax_low_final} with \eqref{e:Adiff_bnd} implies
\begin{align*}
\|\A^{+} - \A^{*}\|_F^2 
&\leq \frac{3}{4}\|\A - \A^{*}\|_F^2 +  \frac{m^2\|P_{\Xi_{\A}}\ee\|_2^2}{n(1 - \tau)} . 
\end{align*}
It remains to bound the error term. For this, we apply Lemma \ref{l:noise_bnd}, which says that there exist constants $c_3, c_4 > 0$ such that 
\[ \|P\ee\|_2^2  \leq c_3\log(n)^2 \sigma^2 m^2(d+1)\]
for all $\A$ satisfying \eqref{e:initbnd} 
with probability greater than $1 - e^{-c_{4} (d+1)m^2 \log(n)}$.

This implies that for $n \geq c_1\tau^{-2}(d+1)m^{10}$, with probability $1 - 4e^{-c_2\tau^{2}n/m^{10}} - n^{-c_4m^2(d+1)}$,
\begin{align*}
\|\A^{+} - \A^{*}\|_F^2 
&\leq \frac{3}{4}\|\A - \A^{*}\|_F^2 +  \frac{c_3 m^4(d+1)\sigma^2\log(n)^2}{n(1 - \tau)}. 
\end{align*}
We now show that given the above upper bound, $\A^{+}$ also satisfies \eqref{e:initbnd}. Indeed, for 
\[n \geq  4\cdot 16 m^2(d+1)
\frac{1}{\kappa^2}\left(\frac{1 + \tau}{1 - \tau}\right)\frac{c_3 m^4(d+1)\sigma^2\log(n)^2}{(1-\tau)},\]
we have
\begin{align*}
    \frac{c_3 m^4(d+1)\sigma^2\log(n)^2}{n(1-\tau)} \leq \frac{\kappa^2}{4}\cdot \frac{1}{16m^2(d+1)}\left(\frac{1 - \tau}{1 + \tau}\right)
\end{align*}
and thus
\begin{align*}
 \|\A^{+} - \A_*\|_F^2 &\leq \frac{3}{4}\|\A - \A_*\|_F^2 + \frac{c_3 m^4(d+1)\sigma^2\log(n)^2}{n(1-\tau)} \leq \frac{\kappa^2}{16m^2(d+1)}\left(\frac{1 - \tau}{1 + \tau}\right). 
\end{align*}

The final conclusion follows from the fact that after $t$ iterations, applying the bound \eqref{e:finalbnd} $t$ times gives
\begin{align*}
   \|\A^{(t)} - \A_*\|_F^2 &\leq \left(\frac{3}{4}\right)^t\|\A^{(0)} - \A_*\|_F^2 + \frac{c_3m^4(d+1)\sigma^2\log(n)^2}{n(1 - \tau)}\sum_{k=0}^{\infty} \left(\frac{3}{4}\right)^k  \\
   &\leq \left(\frac{3}{4}\right)^t\|\A^{(0)} - \A_*\|_F^2 + \frac{c_4 m^4(d+1)\sigma^2\log(n)^2}{n(1 - \tau)},
\end{align*}
and all $t$ bounds hold simultaneously with probability at least $1 - 4e^{-c_2\tau^{-2}n/m^{10}} - n^{-c_4m^2(d+1)}$. 
\end{proof}

 \section{Numerical Experiments}\label{s:examples}
 
In this section, we empirically compare spectrahedral and polyhedral regression for estimating a convex function from data. 
More specifically, we compare $(m,m)$-spectrahedral estimators to $m(m+1)/2$-polyhedral estimators, both of which have $m(m+1)/2$ degrees of freedom per dimension. For each experiment, we apply the alternating minimzation algorithm with multiple random initializations, and the solution that minimizes the least squared error is selected. We adapted the code \cite{SohGithub} for support function estimation used in \cite{Soh19} for spectrahedral regression. 
 
 \subsection{Synthetic Regression Problems}
 
 The first experiments use synthetically generated data from a known convex function, one from a spectrahedral function and another from a convex function that is neither polyhedral nor spectrahedral.
 In both problems below, the root-mean-squared error (RMSE) is obtained by first obtaining estimators form 200 noisy training data points and then evaluating the RMSE of the estimators on 200 test points generated from the true function. We ran the alternating minimization algorithm with 50 random initializations for 200 steps or until convergence, and chose the best estimator. 

 First, we consider $n$ i.i.d. data points distributed as $(X,Y)$, where $X \in \RR^2$ is uniformly distributed in $[-1,1]^2$, and
 \begin{align}\label{e:l2norm}
 Y = \sqrt{X_1^2 + X_2^2} + \ee,
 \end{align}
 where $\ee \sim \mathcal{N}(0, 0.1^2)$. In Figure \ref{f:l2norm}, we have plotted polyhedral and spectrahedral estimators obtained from $n = 20, 50$ and $200$ data points. The RMSE for both models is given in Table \ref{ta:RSME}. The function $y = \|x\|_2$ is a spectrahedral function, and the spectrahedral estimator performs better than the polyhedral estimator as expected.
 
 \begin{figure}[h!]
     \centering
     \includegraphics[width=.9\textwidth]{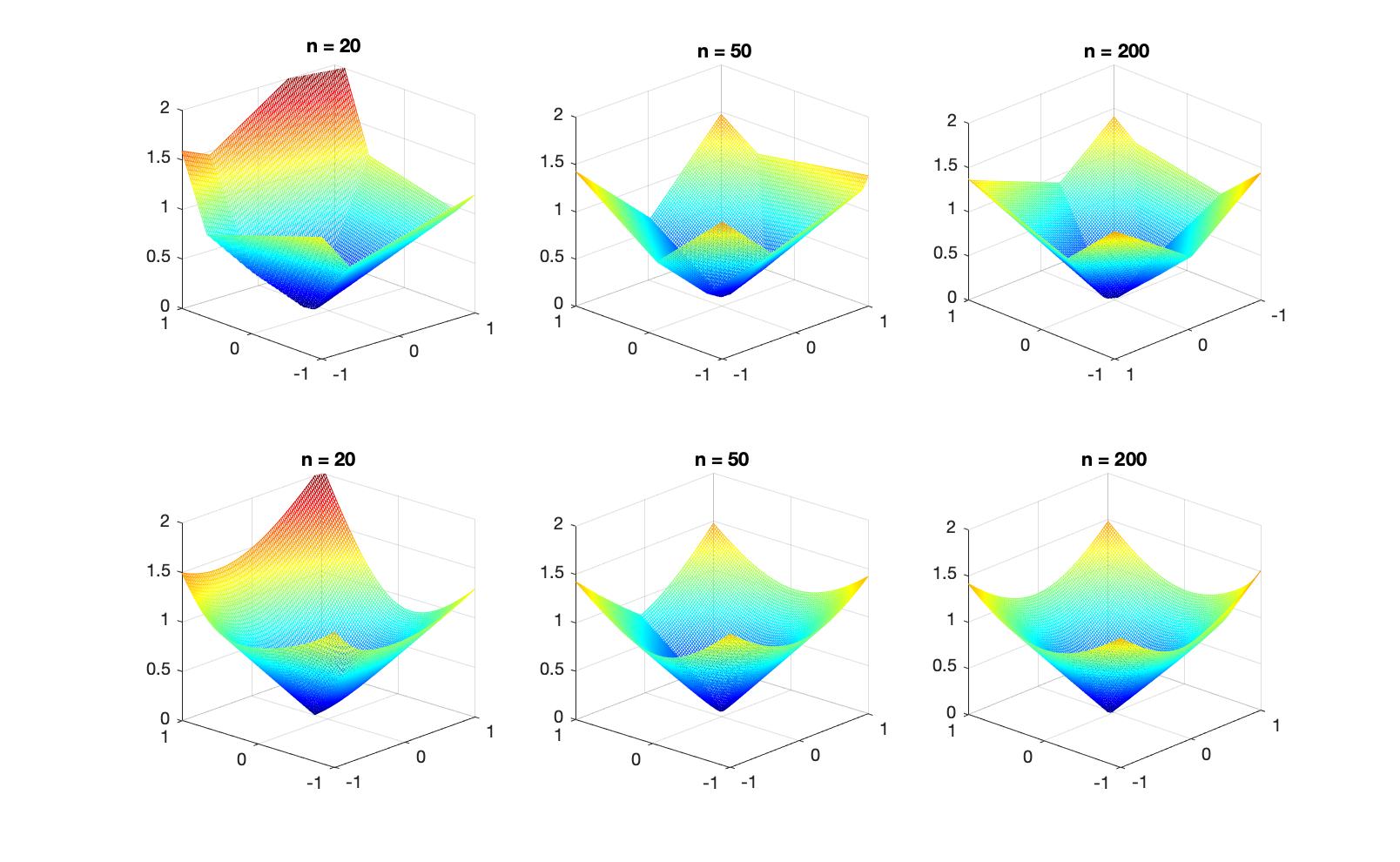}
     \caption{Polyhedral ($m = 6$) and Spectrahedral ($m= 3$) reconstructions of the convex function $y = \|x\|_2$ from $n = 20$, $50$, and $200$ data points from model \eqref{e:l2norm}.}
     \label{f:l2norm}
 \end{figure}
 
 


Second, we consider $n$ i.i.d. data points generated as $(X,Y) \in \RR \times \RR$, where $X \sim \mathcal{N}(0,1)$ and 
\begin{align}\label{e:exp}
Y = \exp(b X) + \ee,
\end{align}
where $b = 1.1394$ 
and $\ee \sim \mathcal{N}(0,0.1^2)$. The underlying convex function is neither polyhedral nor spectrahedral, but the spectrahedral estimator better captures the smoothness of the function as illustrated in Figure \ref{f:expplots}. The spectrahedral estimator also outperforms the polyhedral estimator with respect to the RMSE when comparing the model fitted to the training data set to the test data set, see Table \ref{ta:RSME}.

 \begin{figure}[h!]
     \centering
     \includegraphics[width=.8\textwidth]{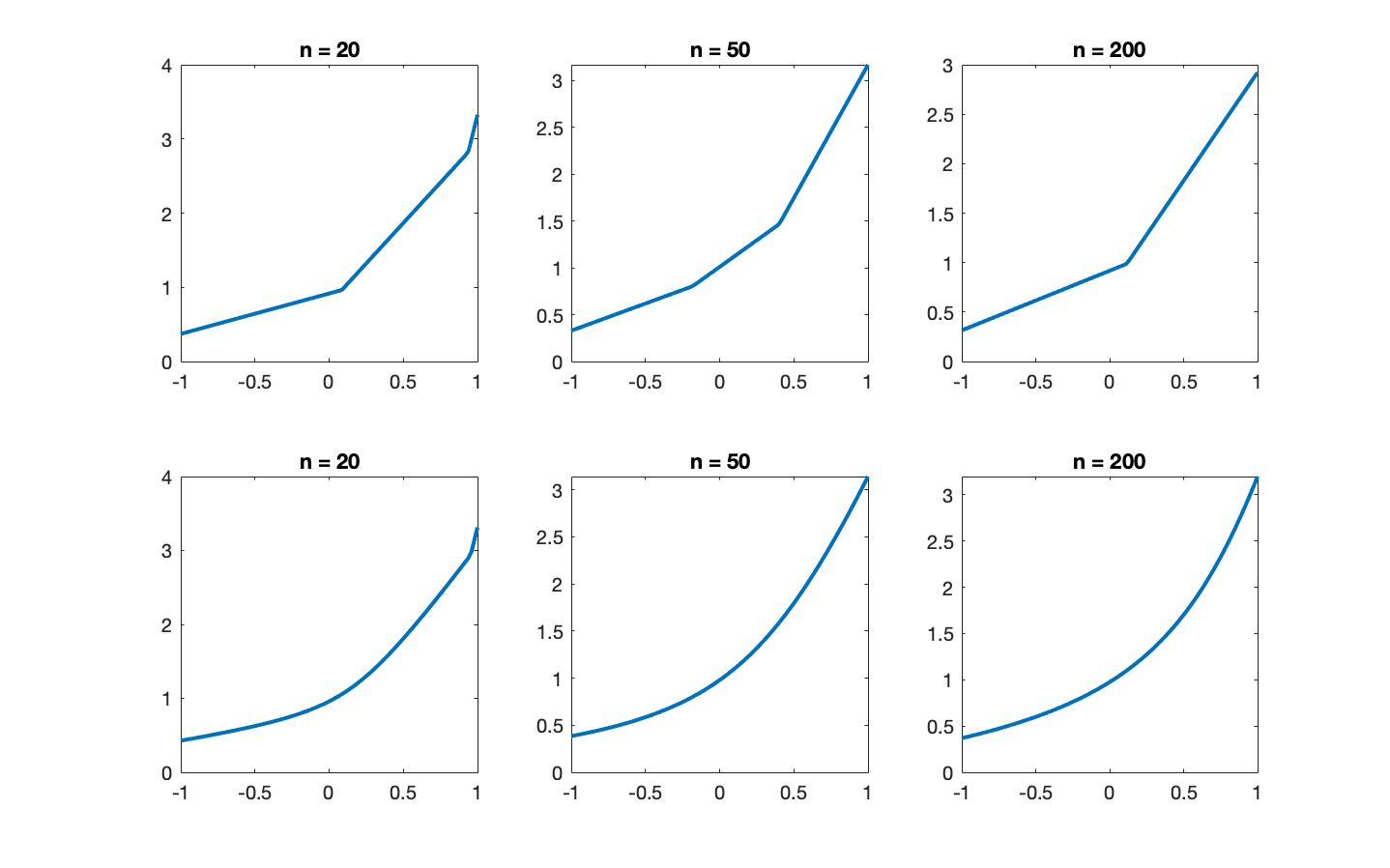}
     \caption{Polyhedral ($m = 6$) and Spectrahedral ($m= 3$) reconstructions of the convex function $y = \exp(\langle x,b\rangle)$ from $n = 20$, $50$, and $200$ noisy data points from model \eqref{e:exp}.}
     \label{f:expplots}
 \end{figure}
 
 \begin{table}
  \caption{RSME for polyhedral and spectrahedral estimators of $y = \exp(bx)$ from model \eqref{e:exp} as $m$ increases.}\label{ta:RSME}
  \centering
    \begin{tabular}{c| c | c | c } 
    Model & $m(m+1)/2$ & Spectrahedral & Polyhedral \\
 \midrule\midrule
    \multirow{3}{*}{\eqref{e:l2norm}} & 3 & 0.0183 & 0.1332 \\
    & 6 & 0.0207 & 0.0416 \\
    &  10 & 0.0243 &  0.0362 \\ \midrule
         \multirow{3}{*}{\eqref{e:exp}} & 3 & 0.1098 & 0.2281\\   & 6 & 0.0902 & 0.1153 \\
      & 10 & 0.0793 & 0.0902\\ \midrule
    \end{tabular}

  \end{table}
 

\subsection{Predicting Average Weekly Wages}

The first experiment we perform on real data is predicting average weekly wages based on years of education and
experience. This data set is also studied in \cite{HannahDunson2013}. The data set is from the 1988 Current Population Survey (CPS) and can be obtained as the data set ex1029 in the Sleuth2 package in R. It consists of 25,361 records of weekly wages for full-time, adult, male workers for 1987, along
with years experience and years of education. 
It is reasonable to expect that wages are concave with respect the years experience. Indeed, at first wages increase with more experience, but with a decreasing return each year until a peak of earnings is reached, and then they begin to decline. 
Wages are also expected to increase as the number of years of education increases, but not in a concave way. However, as in \cite{HannahDunson2013}, we
use the transformation $1.2^{\text{years education}}$ to obtain a concave relationship. 
We used polyhedral and spectrahedral regression to fit convex functions to this data set, as illustrated in Figure \ref{f:3dWage}. We also estimated the RMSE for different values of $m(m+1)/2$ (the degrees of freedom per dimension) through hold-out validation with 20\% of the data points, see Table \ref{ta:RSMEReal}. This generalization error is smaller for the spectrahedral estimator than the polyhedral estimator in each case.

 
 
 


\subsection{Convex Approximation in Engineering Applications}

In the following two examples, we consider applications of convex regression in engineering applications where the goal is to subsequently use the convex estimator as an objective or constraint in an optimization problem. Polyhedral regression returns a convex function compatible with a linear program, and using spectrahedral regression provides an estimator compatible with semidefinite programming.

\subsubsection{Aircraft Data}

In this experiment, we consider the XFOIL aircraft design problem studied in \cite{Hoburgetal2015}. The profile drag on an airplane wing is described by a coefficient CD that is a function of the Reynolds number (Re), wing thickness ratio ($\tau$), and lift coefficient (CL).  
There is not an analytical expression for this relationship, but it can be simulated using XFOIL \cite{XFOIL}. For a fixed $\tau$, after a logarithmic transformation, the data set can be approximated well by a convex function. We fit both spectrahedral and polyhedral functions to this data set, and 
the best fits for the whole data set appears in Figure \ref{f:drag} for models with 6 degrees for freedom per dimension. Then, we performed hold-out validation, training on 80\% of the data and testing on the remaining 20\%. The RMSE is given in Table \ref{ta:RSMEReal}, where we observe that the spectrahedral estimator achieves a smaller error than polyhedral regression. 

 \begin{figure}[h!]
     \centering
     \includegraphics[width=.9\textwidth]{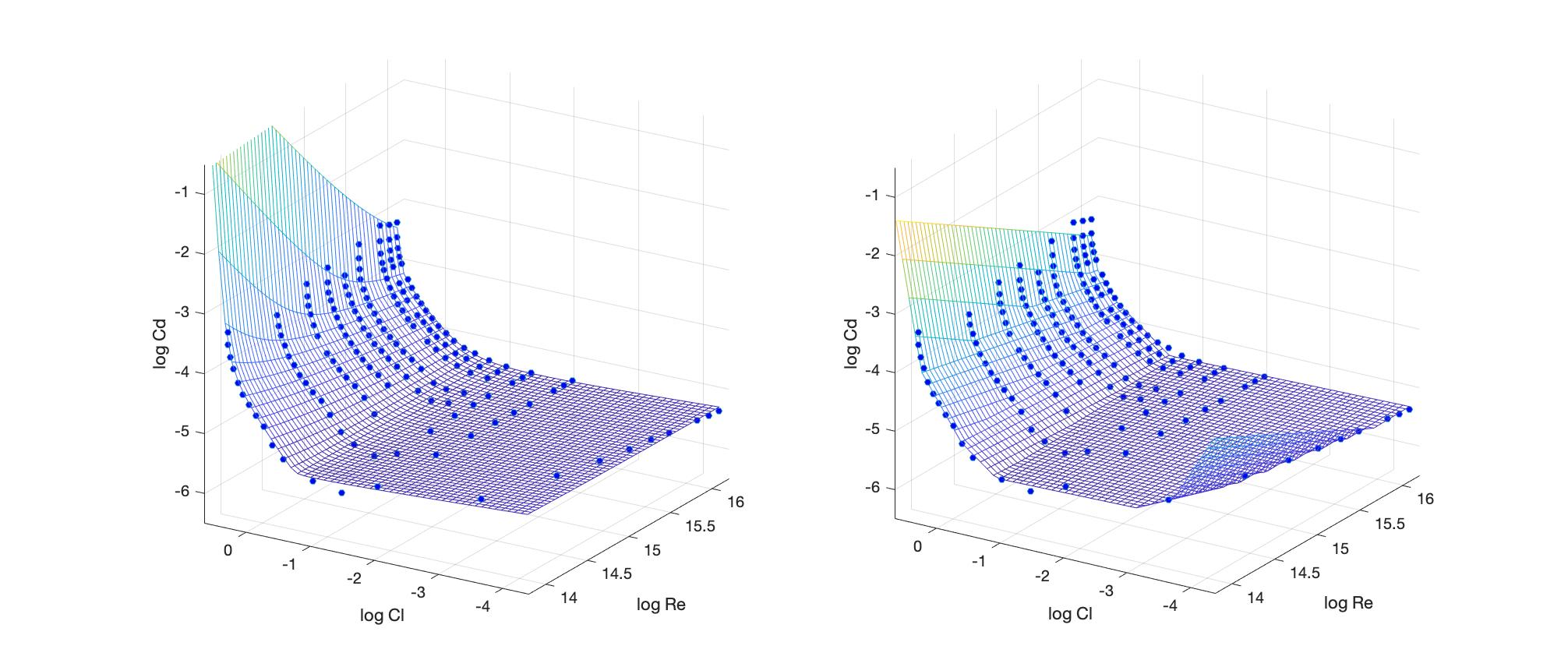}
     \caption{Spectrahedral ($m= 3$) and Polyhedral ($m = 6$) estimators of the log of drag coefficient vs log of Reynolds number and lift coefficient for a fixed thickness ratio $\tau = 8\%$.}
     \label{f:drag}
 \end{figure}

\subsubsection{Power Modeling For Circuit Design}

A circuit is an interconnected collection of electrical components including batteries, resistors, inductors, capacitors, logical gates, and transistors. In circuit design, the goal is to optimize over variables such as devices, gates, threshold, and power supply voltages in order to minimize circuit delay or physical area. 
The power dissipated, $P$, is a function of gate supply $V_{dd}$ and threshold voltages $V_{th}$. The following model, see \cite{Hoburgetal2015} and \cite{HannahDunson2012}, can be used to study this relationship:
\begin{align*}
    P = V_{dd}^2 + 30V_{dd}e^{-(V_{th} - 0.06 V_{dd})/0.039}.
\end{align*}
We generate $n$ i.i.d. data points as in \cite{HannahDunson2012} as follows. For each input-output pair, first sample $u = (V_{dd}, V_{th})$ uniformly over the domain $1.0 \leq V_{dd} \leq 2.0$ and $0.2 \leq V_{th} \leq 0.4$ and compute $P(u)$. Then, apply the transformation $(x,y) = (\log u, \log P(u))$. We fit this collection of transformed data points using polyhedral and spectrahedral regression, and the estimators for $n = 20, 50$, and $200$ are illustrated in Figure \ref{f:CircuitPlots}. We also perform hold-out validation with 20\% of the data for the case $n = 200$ and the RMSE appears in in Table \ref{ta:RSMEReal}. By this measure, the spectrahedral estimator performs much better than the polyhedral estimator in this application. 

\begin{figure}[h!]
     \centering
     \includegraphics[width=.9\textwidth]{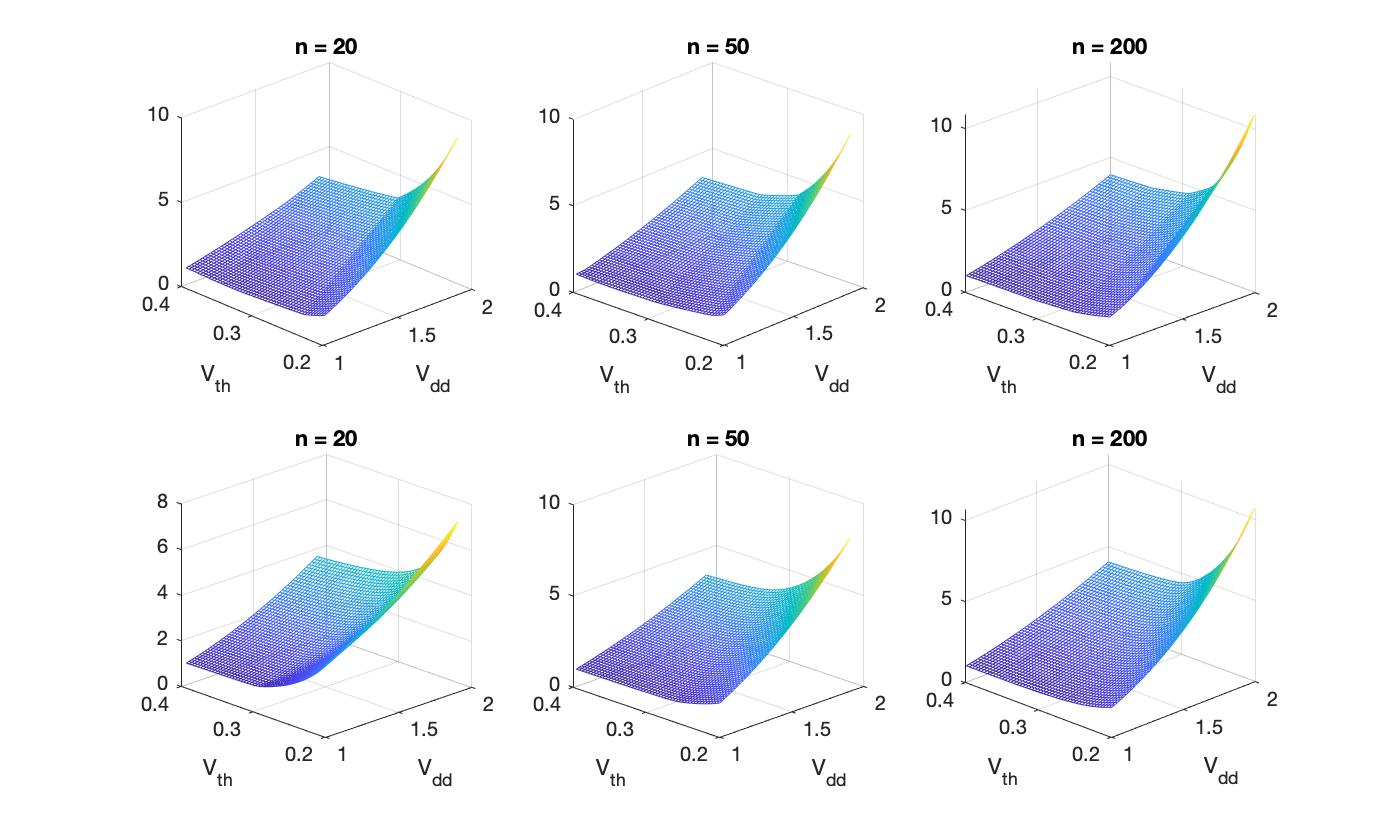}
     \caption{Polyhedral ($m = 6$) and Spectrahedral ($m= 3$) estimators of $n = 20, 50,$ and $200$ transformed data points generated from the power dissipation model.}
     \label{f:CircuitPlots}
 \end{figure}
 
 
 \begin{table}
  \caption{RSME for polyhedral and spectrahedral estimators for real data and engineering experiments.}\label{ta:RSMEReal}
  \centering
    \begin{tabular}{c| c | c | c } 
    Application & $m(m+1)/2$ & Spectrahedral & Polyhedral \\
 \midrule\midrule
    \multirow{3}{*}{Average Weekly Wages} & 3 & 142.1166 & 145.5803 \\
     & 6 & 140.1173 & 141.4989 \\
      & 10 & 140.0352 & 141.9851 \\ \midrule
         \multirow{3}{*}{Aircraft Profile Drag} & 3 & 0.086 & 0.0895 \\   & 6 & 0.0576 & 0.0709 \\
      & 10 & 0.0452 & 0.0515\\ \midrule
      \multirow{3}{*}{Circuit Design} & 3 & 0.0085 & 0.02 \\   & 6 & 0.0072 & 0.012 \\
      & 10 & 0.0072 & 0.0088 \\ \midrule
    \end{tabular}
  \end{table}

 \section{Discussion and Future Work}\label{sec:discussion}
 
 In this work, we have introduced spectrahedral regression as a new method for estimating a convex function from noisy measurements. Spectrahedral estimators are appealing from a qualitative and quantitative perspective and we have shown they hold advantages over the usual LSE methods as well as polyhedral estimators when the underlying convex function is non-polyhedral. 
 Our theoretical results and numerical experiments call for further study of the expressivity of this model class and its computational advantages. We now describe a few directions of future research.
 
 An interesting open question is to obtain the approximation rate for $(m,k)$-spectrahedral functions to the class of Lipschitz convex functions and $(m,k)$- spectratopes to the class of convex bodies for general $k$. There is extensive literature on this approximation question for polytopes (see, for instance, \cite{Bronshtein2008, Dudley74}), and we have obtained matching bounds (up to logarithmic factors) for fixed $k >1$. For $k$ depending on $m$, and in particular in the case $k = m$, the literature is more limited, one example is \cite{Barvinok2012}. Progress in this direction would complete our understanding of the expressive power of the model presented here and have important consequences for how well semidefinite programming can approximate a general convex program.
 
We have also proved computational guarantees for a natural alternating minimization algorithm for spectrahedral regression. However, this convergence guarantee depends on a good initialization. In practice, running the algorithm with multiple random initializations and taking the estimator with the smallest error works well, but it would be very interesting to extend the results on initialization in \cite{Ghoshetal2020Gaussian} to the spectrahedral case. 
 Another line of future work is to extend other methods to solve the non-convex optimization \eqref{eq:speclse} in the max-affine case such as the adaptive partitioning method in \cite{HannahDunson2013} and the method proposed in \cite{Balazs2015}. These algorithms also lack theoretical guarantees and it would be interesting to obtain conditions under which these methods obtain good estimates of the true parameter. 
 
\section*{Acknowledgments}
E. O. was supported by NSF MSPRF Award 2002255 with additional funding from ONR Award N00014-18-1-2363. V. C. was supported in part by National Science Foundation grant
CCF-1637598, in part by National Science Foundation grant DMS-2113724,
and in part by AFOSR grant FA9550-20-1-0320. 

\bibliographystyle{siamplain}
\bibliography{biblio}
 
 \appendix

\section{Lemmas for the Proof of Theorem \ref{t:conv_bnd}}

We first give a few definitions that are needed in following lemmas. A random vector $\xi \in \RR^d$ is sub-Gaussian with parameter $\eta$ if $\EE[X] = 0$ and for each $u \in \mathbb{S}^{d-1}$, $\EE\left[e^{\lambda\langle u, X\rangle}\right] \leq e^{\lambda^2\eta^2/2}, \quad \text{ for all } \lambda \in \RR$.
The sub-Gaussian norm of a random variable $X$, denoted by $\|X\|_{\psi_2}$, is defined as
\[\|X\|_{\psi_2}= \inf\{t > 0 : \EE[\exp(X^2/t^2)] \leq 2\}.\]
For $\xi \in \RR^d$, the sub-Gaussian norm is defined as
$\|\xi\|_{\psi_2} := \sup_{u \in \mathbb{S}^{d-1}}\|\langle \xi, u \rangle \|_{\psi_2}$.
The sub-exponential norm of $X$, denoted 
by $\|X\|_{\psi_1}$, is defined as
\[\|X\|_{\psi_1}= \inf\{t > 0 : \EE[\exp(|X|/t)] \leq 2\},\]
and the sub-exponential norm of a vector is defined similarly.

We also recall that the covering number of a Euclidean ball satisfies
\begin{align}\label{e:cov_bnd}
\mathcal{N}(B_q(z,R), \|\cdot\|_2, \ee) \leq (1 + 2R/\ee)^q
\end{align}
for $\ee \leq 2r$ by a standard volume argument. 

The proofs rely on uniform spectral concentration bounds of a sample covariance matrix, which follow from Bernstein's inequality and Dudley's inequality. A general reference for the ideas in the lemmas below is \cite{VershyninBook}.



\begin{lemma}\label{l:spectrum_bnd}
Let $\xi$ be an $\eta$-sub-Gaussian r.v. in $\RR^d$ 
and $\A_* \in (\mathbb{S}^m)^d$ be such that 
\begin{align*}
    \inf_{u \in \mathbb{S}^{d-1}} \lambda_1(\A_*[u]) - \lambda_2(\A_*[u]) := \kappa > 0.
\end{align*}
Define the set $B(\A_*,\kappa/4) := \{\A \in (\mathbb{S}^m)^d :   \|\A - \A_{*}\|_F \leq \frac{\kappa}{4}\}$. 
For each $\A \in (\mathbb{S}^m)^d$, define $U_{\A}$ to be the rank one matrix such that 
\[\langle \xi \otimes U_{\A}, \A \rangle = \langle U_{\A}, \A\xi \rangle = \lambda_{\max}(\A\xi).\] Then, $\|\xi \otimes U_{\A}\|_{\psi_2} \leq \eta m^2$ and for all $\A_1, \A_2 \in B(\A_*, r)$, 
\begin{align*}
\|\xi \otimes U_{\A_1} - \xi \otimes U_{\A_2}\|_{\psi_2} \leq \frac{8m^2}{\kappa}\|\A_1 - \A_2\|_F.
\end{align*}


\end{lemma}

\begin{proof}
For the first claim, recall that $\xi \otimes U_{\A}$ is sub-Gaussian if $\langle \xi \otimes U_{\A}, v \rangle$ is sub-Gaussian for every $v \in \mathbb{S}^{dm^2-1}$. Indeed, we first see that for $v \in \mathbb{S}^{dm^2-1}$,
\begin{align*}
 \left|\langle \xi \otimes U, v \rangle\right| = \left|\sum_{i=1}^d \sum_{k=1}^{m^2} \xi_iU_k v_{ik}\right| \leq \sum_{k=1}^{m^2} |\langle \xi, v^{(k)} \rangle| |U_k| \leq \sum_{k=1}^{m^2} |\langle \xi, v^{(k)} \rangle|,
\end{align*}
where $v^{(k)} = (v_{1k}, \ldots, v_{dk}) \in \RR^d$.
Then, by the triangle inequality,
\begin{align*}
\left\|\langle \xi \otimes U, v \rangle\right\|_{\psi_2} \leq \left\|\sum_{k=1}^{m^2} |\langle \xi, v^{(k)} \rangle|\right\|_{\psi_2} \leq \sum_{k=1}^{m^2} \left\||\langle \xi, v^{(k)} \rangle|\right\|_{\psi_2} \leq \eta \sum_{k=1}^{m^2} \|v^{(k)}\|_2 \leq \eta m^2.
\end{align*}
For the second claim, first note that for all $\A \in B(\A_*,\kappa/4)$, Weyl's inequality implies
\begin{align*}
     \lambda_1(\A[u]) - \lambda_2(\A[u]) \geq \lambda_1(\A_*[u]) - \lambda_2(\A_*[u]) - 2\|\A - \A_{*}\|_{op} \geq \frac{\kappa}{2} > 0.
\end{align*}
Then, observe that $\|U_1 - U_2\|^2_F \leq 2 \|u_1 - u_2\|_2^2$, where $U_1 = u_1 u_1^T$, $U_2 = u_2 u_2^T$ and $u_1, u_2 \in \mathbb{S}^{m-1}$. By a variation of the  Davis-Kahan Theorem (Corollary 3 in \cite{WangSamworth2015}),
\begin{align*}
    \|U_1-U_2\|_F \leq \sqrt{2} \|u_1-u_{2}\|_2 &\leq \frac{4\|(\A_1 - \A_2)[\xi]\|_{op}}{\lambda_1(\A_2[\xi]) - \lambda_2(\A_2[\xi])} \leq \frac{8}{\kappa}\|\A_1 - \A_2\|_F.
\end{align*}
This implies that 
\begin{align*}
   \left\|\xi \otimes (U_1 - U_2)\right\|_{\psi_2} \leq  \frac{8\|\A_1 - \A_2\|_F}{\kappa}\left\|\xi \otimes \frac{U_1 - U_2}{ \|U_1-U_2\|_F} \right\|_{\psi_2} \leq \frac{8\eta m^2}{\kappa}\|\A_1 - \A_2\|_F,
\end{align*}
\end{proof}

\begin{lemma}\label{l:unif_cov}
Define $B_q(z, R) := \{x \in \RR^q : \|x - z\|_2 \leq R\}$ for $R > 0$ and $z \in \RR^q$. Let $\{\xi_{a}\}_{a \in B_q(z,R)}$ be stochastic process in $\RR^d$ such that 
\begin{itemize}
    \item[(i)] $\|\xi_a \|_{\psi_2} \leq \eta$;
    \item[(ii)] for all $a_1$, $a_2 \in B_q(z,R)$,
   $\|\xi_{a_1} - \xi_{a_2}\|_{\psi_2} \leq K \|a_1 - a_2\|_2$.
\end{itemize}
Define $\Xi_a \in \RR^{N \times d}$ to be the matrix with $N$ i.i.d. rows in $\RR^d$ distributed as $\xi_a$, and let $\Sigma_a := \EE[\xi_a \xi_a^T]$. Fix $\tau \in (0,1)$. Then, there exist absolute constants $c_0, c_1, c_2 > 0$ such that if $n \geq c_0\tau^{-2}K^2\eta^4R^2\max\{q, d\}$,
\begin{align*}
    \PP\left(\sup_{a \in B_q(z,R)} \left\| \frac{1}{n} \Xi_a^T \Xi_a - \Sigma_a\right\|_{op} \geq \tau\|\Sigma_a\|_{op} \right)
    &\leq e^{-c_1 n \tau^{2}/K^2\eta^4R^2}.
\end{align*}
This implies that with probability greater than $1 -  e^{-c_1 n \tau^{2}/K^2\eta^4R^2}$,
\begin{align*}
   \lambda_{\max}( \Sigma_a)(1 - \tau) \leq \inf_{a \in B_q(z,R)}   \frac{\lambda_{\min}\left(\Xi_{a}^T \Xi_{a}\right)}{n}\leq \sup_{a \in B_q(z,R)}  \frac{\lambda_{\max}\left(\Xi_{a}^T \Xi_{a}\right)}{n} \leq \lambda_{\max}( \Sigma_a)(1 + \tau).
\end{align*}
\end{lemma}

\begin{proof}
First suppose that for all $a$, $\Sigma_a = I$, i.e. that is $\xi_a$ is isotropic. For the general case, since $\Sigma_a^{-1/2} \xi_a$ is isotropic, the conclusion follows from the fact that $\left\| \frac{1}{n} \Xi_a^T \Xi_a - \Sigma_a\right\|_{op} \leq  \|\Sigma_a\|_{op} \left\| \frac{1}{n} \Sigma_a^{-1}\Xi_a^T \Xi_a - I\right\|_{op}$. 

We first show that for any $x \in \mathbb{S}^{d-1}$, the stochastic process 
$X_a := \frac{1}{\sqrt{n}}\|\Xi_a x\| - 1$ 
has sub-Gaussian increments 
$\|X_{a_1} - X_{a_2}\|_{\psi_2} =  \frac{1}{\sqrt{n}} \left\|\|\Xi_{a_1} x\|_2 - \|\Xi_{a_2} x\|_2 \right\|_{\psi_2}$.

\textbf{Case 1:} $s \in \left[0, 4 K\sqrt{n}\right]$. We first see that
\begin{align}\label{e:tailbound}
   & \PP\left(|\|\Xi_{a_1} x\|_2 - \|\Xi_{a_2} x\|_2| \geq s\|a_1 - a_2\|_2\right) \nonumber \\
   & \qquad \qquad = \PP\left(\frac{\left|\|\Xi_{a_1} x\|_2^2 - \|\Xi_{a_2} x\|_2^2\right|}{\|a_1 - a_2\|} \geq s(\|\Xi_{a_1} x\|_2 + \|\Xi_{a_2} x\|_2)\right) \nonumber\\
   &\qquad \qquad \leq  \PP\left(\frac{\left|\|\Xi_{a_1} x\|_2^2 - \|\Xi_{a_2} x\|_2^2\right|}{\|a_1 - a_2\|} \geq s \|\Xi_{a_1} x\|_2\right) \nonumber \\
  & \qquad \qquad \leq \PP\left(\frac{\left|\|\Xi_{a_1} x\|_2^2 - \|\Xi_{a_2} x\|_2^2\right|}{\|a_1 - a_2\|} \geq  \frac{s\sqrt{n}}{2}\right)  + \PP\left(\|\Xi_{a_1} x\|_2 \leq \frac{\sqrt{n}}{2}\right) \nonumber\\
   &\qquad \qquad \leq \PP\left(\frac{\left|\|\Xi_{a_1} x\|_2^2 - \|\Xi_{a_2} x\|_2^2\right|}{\|a_1 - a_2\|} \geq \frac{s\sqrt{n}}{2}\right)  + \PP\left(\left|\|\Xi_{a_1} x\|_2 - \sqrt{n} \right| \geq \frac{s}{8 K} \right).
\end{align}
Then, note that
\begin{align*}
\|\Xi_{a_1} x\|_2^2 - \|\Xi_{a_2} x\|_2^2 
= \sum_{i=1}^n \langle \xi_{a_1}^{(i)}- \xi_{a_2}^{(i)}, x \rangle \langle \xi_{a_1}^{(i)}+ \xi_{a_2}^{(i)}, x \rangle,
\end{align*}
and by Lemma 2.7.7 in \cite{VershyninBook}, 
\begin{align*}
    \|\langle \xi_{a_1}^{(i)}- \xi_{a_2}^{(i)}, x \rangle \langle \xi_{a_1}^{(i)}+ \xi_{a_2}^{(i)}, x \rangle\|_{\psi_1} &\leq \|\langle \xi_{a_1}^{(i)}- \xi_{a_2}^{(i)}, x \rangle  \|_{\psi_2} \|\langle \xi_{a_1}^{(i)}+ \xi_{a_2}^{(i)}, x \rangle\|_{\psi_2}\\ 
    &\leq 2 \eta K\|a_1 - a_2\|_{2}.
\end{align*}
Each term in the sum also has zero mean. Indeed,
\begin{align*}
    \EE[\langle \xi_{a_1}^{(i)}- \xi_{a_2}^{(i)}, x \rangle \langle \xi_{a_1}^{(i)}+ \xi_{a_2}^{(i)}, x \rangle] &= \EE[\langle \xi_{a_1}^{(i)}, x \rangle^2 - \langle \xi_{a_2}^{(i)}, x \rangle^2] = 0.
\end{align*}
Applying Bernstein's inequality (Corollary 2.8.3 in \cite{VershyninBook}) gives for all $t \geq 0$,
\begin{align*}
   \PP\left(\frac{\|\Xi_{a_1} x\|_2^2 - \|\Xi_{a_2} x\|_2^2}{\|a_1 - a_2\|} \geq t\right) \leq  2e^{-c_1\min\left\{\frac{t^2}{4\eta^2K^2n}, \frac{t}{2\eta K} \right\}}. 
\end{align*}
For the second tail probability in \eqref{e:tailbound}, Theorem 3.1.1 in \cite{VershyninBook} implies 
\begin{align*}
   \PP\left(\left|\|\Xi_{a_1} x\|_2 - \sqrt{n} \right| \geq t \right) \leq 2e^{- \frac{c_2 t^2}{\eta^4}},
\end{align*}
where we have used that $\xi_a$ is isotropic. Thus, since $s < 4K\sqrt{n}$ and $\eta \geq 1$,
\begin{align*}
\PP\left(\frac{|\|\Xi_{a_1} x\|_2 - \|\Xi_{a_2} x\|_2|}{\|a_1 - a_2\|_2} \geq s\right)
   &\leq 2e^{- c_1\min\{\frac{s^2}{16\eta^2K^2}, \frac{s\sqrt{n}}{4\eta K}\}} + 2e^{-\frac{c_2 s^2}{64\eta^4K^2}} \leq 4e^{-\frac{c_3s^2}{\eta^4K^2}}.
\end{align*}

\textbf{Case 2}: $s \geq 4\eta K\sqrt{n}$. 
By the triangle inequality,
\begin{align*}
    & \PP\left(\frac{|\|\Xi_{a_1} x\|_2 - \|\Xi_{a_2} x\|_2|}{\|a_1 - a_2\|} \geq s\right) 
    \leq \PP\left(\frac{\|(\Xi_{a_1} -\Xi_ {a_2}) x\|^2}{\|a_1 - a_2\|^2} \geq s^2\right) \\
    &= \PP\left(\frac{\|(\Xi_{a_1} -\Xi_ {a_2}) x\|^2}{\|a_1 - a_2\|^2} - n\frac{\EE[\langle \xi_{a_1} - \xi_{a_2}, x \rangle^2]}{\|a_1 - a_2\}^2} \geq s^2 - n\frac{\EE[\langle \xi_{a_1} - \xi_{a_2}, x \rangle^2]}{\|a_1 - a_2\|^2}\right) \\
    &\leq \PP\left(\frac{\|(\Xi_{a_1} -\Xi_ {a_2}) x\|^2}{\|a_1 - a_2\|^2} - n\frac{\EE[\langle \xi_{a_1} - \xi_{a_2}, x \rangle^2]}{\|a_1 - a_2\|^2} \geq s^2 - 4K^2 n\right) \\
    &\leq \PP\left(\left|\frac{\|(\Xi_{a_1} -\Xi_ {a_2}) x\|^2}{\|a_1 - a_2\|^2} - n\frac{\EE[\langle \xi_{a_1} - \xi_{a_2}, x \rangle^2]}{\|a_1 - a_2\|^2}\right| \geq \frac{3s^2}{4}\right).
\end{align*}
where for the second to last inequality we have used that
\[\EE[\langle \xi_{a_1} - \xi_{a_2}, x \rangle^2]\leq 4\|\xi_{a_1} - \xi_{a_2}\|^2_{\psi_2} \leq 4K^2\|a_1 - a_2\|^2_2,\]
and the last inequality follows from the lower bound on $t$ and the fact that $\eta \geq 1$.
By Bernstein's inequality again (Corollary 2.8.3 in \cite{VershyninBook}) and the lower bound on $t$,
\begin{align*}
\PP\left(\frac{|\|\Xi_{a_1} x\|_2 - \|\Xi_{a_2} x\|_2|}{\|a_1 - a_2\|} \geq s\right) \leq 2e^{-c_4 \min\left\{\frac{s^4}{nK^4}, \frac{s^2}{K^2}\right\}} \leq 2e^{-\frac{c_4s^2}{K^2}} 
\end{align*}

Proposition 2.5.2 in \cite{VershyninBook} then implies that for all $t \geq 0$, 
\begin{align*}
   \|X_{a_1} - X_{a_2}\|_{\psi_2} \leq \frac{K\eta^2}{\sqrt{n}} \|a_1 - a_2\|_2. 
\end{align*}

Now, by Theorem 8.1.6 in \cite{VershyninBook} and \eqref{e:cov_bnd}, for $\delta > 0$
and $n \geq q \delta^{-2}$, 
\begin{align}\label{e:sup_bnd_1}
    \sup_{a \in B_q(z,R)} \left|\frac{1}{\sqrt{n}} \|\Xi_a x\| - 1 \right| \leq \frac{c_5K\eta^2}{\sqrt{n}}\left(c_6 R\sqrt{q} + 2R \delta \sqrt{n}\right) 
    \leq c_7 K\eta^2 R\delta,
\end{align}
with probability greater than $1 - 2e^{-\delta^2 n}$. Now, let $\tau \in(0,1)$. By the inequality $|z^2 - 1| \leq 3\max\{|z - 1|, |z - 1|^2\}$ for all $z \geq 0$,
\begin{align*}
    \PP\left( \sup_{a \in B_q(z,R)} \left|\frac{1}{n} \|\Xi_a x\|^2_2 - 1 \right| \geq \frac{\tau}{2} \right) &\leq \PP\left( \sup_{a \in B_q(z,R)} \left|\frac{1}{\sqrt{n}} \|\Xi_a x\|_2 - 1 \right| \geq \frac{\tau}{6} \right).
\end{align*}
Letting $\delta = \frac{\tau }{6 c_7 K\eta^2 R}$ in \eqref{e:sup_bnd_1} gives the following. For $n \geq c_8 qK^2\eta^4R^2\tau^{-2}$, 
\begin{align*}
    \PP\left( \sup_{a \in B_q(z, R)} \left|\frac{1}{n} \|\Xi_a x\|^2_2 - 1 \right| \geq \frac{\tau}{2} \right) \leq 2e^{-n\tau^2/c_8K^2\eta^4R^2}.
\end{align*}

Finally, by Lemma 5.3 in \cite{VershyninRandomMatrices},
\begin{align*}
\sup_{a \in B_q(r)} \| \frac{1}{n} \Xi_a^T \Xi_a - I\|_{op} \leq 2 \max_{x \in \mathcal{N}} \sup_{a \in B_q(r)} \left|\frac{1}{n} \|\Xi_a x\|^2_2 - 1\right|,
\end{align*}
where $\mathcal{N}$ is a $\frac{1}{4}$-net of the unit sphere $\mathbb{S}^{d-1}$. Lemma 5.4 in \cite{VershyninRandomMatrices} implies $|\mathcal{N}| \leq 9^d$. Applying the union bound then gives, for $n \geq c_8qK^2\eta^4R^2\tau^{-2}$,
\begin{align*}
    &\PP\left(\sup_{a \in B_q(z,R)} \| \frac{1}{n} \Xi_a^T \Xi_a - I\|_{op} \geq \tau \right) \leq \PP\left( \max_{x \in \mathcal{N}} \sup_{a \in B_q(r)} \left|\frac{1}{n} \|\Xi_a x\|^2_2 - \|x\|^2_2 \right| \geq \frac{\tau}{2} \right) \\
    &\quad \leq |\mathcal{N}|\PP\left(\sup_{a \in B_q(z,R)} \left|\frac{1}{n} \|\Xi_a x\|^2_2 - 1\right| \geq \frac{\tau}{2} \right) 
    \leq 2\cdot 9^de^{-n\tau^{2}/c_8K^2\eta^4R^2},
\end{align*}
Thus, there exist absolute constants $b_1, b_2$ such that for $n \geq b_1 \tau^{-2}K^2\eta^4R^2\max\{q, d\}$,
\begin{align*}
    \PP\left(\sup_{a \in B_q(z,R)} \left\| \frac{1}{n} \Xi_a^T \Xi_a - I\right\|_{op} \geq \tau \right)
    &\leq 2\cdot e^{-b_2 n \tau^{2}/K^2\eta^4R^2}.
\end{align*}
\end{proof} 

\begin{lemma}\label{l:noise_bnd}
Consider the setting of Theorem \ref{t:conv_bnd}. Let $B(\A_{*},\kappa/4) := \{ \A \in (\mathbb{S}^m)^d : \|\A - \A_{*}\|_F^2 \leq \kappa/4\}$ and define $\mathcal{P} := \{P_{\Xi_{\A}} : \mathcal{A} \in \mathcal{B}(\mathcal{A}_*, \kappa/4)\}$. Then, there exist absolute constants $c_0$, $c_1$, and $c_2$ such that for $n \geq c_0$,
\begin{align*}
     \PP\left(\sup_{P \in \mathcal{P}}\|P\ee\|^2  \geq c_1\log(n)^2 \sigma^2m^2(d+1)\right) \leq \exp\left\{ -c_{2} (d+1)m^2 \log(n) \right\}.
\end{align*}
\end{lemma}

\begin{proof}
First, note that for all $P \in \mathcal{P}$, $\|P\|^2_F = (d+1)m^2$ and
\begin{align}\label{e:Pnorm2}
\EE[\|P\ee\|_2^2] = \sum \EE[\langle v_i, \ee \rangle^2] = \sum \|v_i\|_2^2 \sigma^2 = (d+1)m^2 \sigma^2.    
\end{align}
Then, 
\begin{align*}
    \sup_{P \in \mathcal{P}} \left(\|P\ee\|^2 - \EE[\|P\ee\|^2]\right) - \EE \sup_{P \in \mathcal{P}} \left(\|P\ee\|^2 - \EE[\|P\ee\|^2] \right)
    = \sup_{P \in \mathcal{P}}\|P \ee\|^2 - \EE\sup_{P \in \mathcal{P}} \|P \ee\|^2.
\end{align*}
Now, recall that $M := \|\max_{i=1, \ldots, n} \ee_i\|_{\psi_2} \leq c_0\sigma\sqrt{\log n}$ for an absolute constant $c_0$ \cite{TalagrandBook}. Applying Theorem 1.1 in \cite{KlochkovZhivotovskiy2019} to the family of matrices $\{P^TP: P \in \mathcal{P}\}$ gives: For $t \geq \max\{c_0 \sigma\sqrt{\log(n)} \EE\left[\sup_{P \in \mathcal{P}} \|P\ee\|_2\right], c_0^2 \sigma^2 \log(n)\}$, 
\begin{align}\label{e:unif_HW}
    \PP\left(\sup_{P \in \mathcal{P}}\|P \ee\|^2  - \EE\left[\sup_{P \in \mathcal{P}} \|P \ee\|^2\right] \geq t \right) \leq e^{ -\frac{c_2}{\sigma^2 \log(n)}\min\left\{ \frac{t^2}{\EE\left[\sup_{P \in \mathcal{P}} \|P\ee\|\right]^2}, t\right\}}.
\end{align}
Also by \eqref{e:Pnorm2}, $\EE\left[\sup_{P \in \mathcal{P}} \|P\ee\|^2\right] \geq (d+1)m^2\sigma^2$ and thus, 
\begin{align*}
    \PP\left(\sup_{P \in \mathcal{P}}\|P \ee\|^2  - (d+1)m^2\sigma_{\ee}^2 \geq t \right) \leq e^{ -\frac{c_2}{\sigma^2 \log(n)}\min\left\{ \frac{t^2}{\EE\left[\sup_{P \in \mathcal{P}} \|P\ee\|\right]^2}, t \right\}}.
\end{align*}

Letting $t = c_3\sigma^2\log(n)^2 (d+1)m^2$ for a constant $c_3 > 0$ large enough, 
\begin{align}\label{e:unif_HW_2}
     &\PP\left(\sup_{P \in \mathcal{P}}\|P\ee\|^2  \geq (c_3\log(n)^2 + 1)\sigma^2m^2(d+1)\right)\\
     & \qquad \qquad \qquad  \leq e^{-c_4 (d+1)m^2  \log(n) \min\left\{ \frac{\log(n)^2 (d+1)m^2\sigma^2}{\EE\left[\sup_{P \in \mathcal{P}} \|P\ee\|\right]^2}, 1\right\}}.\nonumber
\end{align}
We now upper bound $\EE\left[\sup_{P \in \mathcal{P}} \|P \ee\|\right]$. 
For $P = P_{\Xi_{\mathcal{A}}} \in \mathcal{P}$,
\begin{align*}
    \| P \ee\|_2 = \|\Xi_{\A} (\Xi_{\A}^T\Xi_{\A})^{-1}\Xi_{\A}^T \ee\|_2 \leq \frac{\|\Xi_{\A}\|_2}{\|\Xi_{\A}^T\Xi_{\A}\|_2} \|\Xi_{\A}^T \ee\|_2 = \frac{\|\Xi_{\A}^T \ee\|_2}{\|\Xi_{\A}^T\|_2} = \frac{\|\Xi_{\A}^T \ee\|_2}{\sqrt{\sum_{i=1}^n \|\xi^{(i)}\|^2_2}}.
\end{align*}
Define the stochastic process $X_{\mathcal{A}} := \frac{\|\Xi_{\mathcal{A}}^T \ee\|_2}{\sum_{i=1}^n \|\xi^{(i)}\|_2}$. 
Then, for $\mathcal{A}$ and $\mathcal{B}$ in $(\mathbb{S}^{m})^d$,
\begin{align*}
    \|\Xi_{\mathcal{A}} - \Xi_{\mathcal{B}}\|_F^2 &= \sum_{i=1}^n \sum_{j=1}^d \sum_{k=1}^{m^2} |\xi^{(i)}_{j}U^{(i)}_{k} - \xi^{(i)}_{j}V^{(i)}_{k}|^2  
    =\sum_{i=1}^n \|\xi^{(i)}\|_2^2   \left[2 - 2\langle U^{(i)}, V^{(i)}\rangle \right] \\
    &= 2\sum_{i=1}^n \|\xi^{(i)}\|_2^2   \left[1 - \langle u^{(i)}, v^{(i)}\rangle^2\right]
    = 2\sum_{i=1}^n \|\xi^{(i)}\|_2^2 \sin \Theta(u^{(i)},v^{(i)})^2.
\end{align*}
By a variant of the Davis-Kahan Theorem (Corollary 3 in \cite{WangSamworth2015}) and \eqref{e:eigengap},
\begin{align*}
\sin \Theta(u^{(i)},v^{(i)}) \leq \frac{2\|\mathcal{A}[\xi^{(i)}] - \mathcal{B}[\xi^{(i)}]\|_{op}}{\lambda_1(\mathcal{A}[\xi^{(i)}])- \lambda_2(\mathcal{A}[\xi^{(i)}])} \leq \frac{4}{\kappa}\|\mathcal{A} - \mathcal{B}\|_{op}.   
\end{align*}
Then, $\frac{ \|\Xi_{\mathcal{A}} - \Xi_{\mathcal{B}}\|_F}{\sqrt{\sum_{i=1}^n \|\xi^{(i)}\|_2^2}} \leq  \frac{8}{\kappa}\|\mathcal{A} - \mathcal{B}\|_{op}$,
and by the Hanson-Wright Inequality \cite[Theorem 2.1]{RudelsonVershynin2013}, since $\ee$ is independent of the $\xi^{(i)}$'s, there is a constant $c_4$ such that
\begin{align*}
    &\PP\left(|X_{\A} - X_{\mathcal{B}}| \geq t + \frac{8\sigma}{\kappa}\|\mathcal{A} - \mathcal{B}\|_{op} \right)  \leq \PP\left(\frac{\|(\Xi_{\mathcal{A}}^T - \Xi_{\mathcal{B}}^T)\ee\|_2}{\sqrt{\sum_{i=1}^n \|\xi^{(i)}\|^2_2}} \geq t + \frac{8\sigma}{\kappa}\|\mathcal{A} - \mathcal{B}\|_{op} \right) \\
    & \qquad\leq \PP\left(\left|\frac{\|(\Xi_{\mathcal{A}}^T - \Xi_{\mathcal{B}}^T)\ee\|_2}{\sqrt{\sum_{i=1}^n \|\xi^{(i)}\|_2^2}} - \frac{\sigma\|\Xi_{\mathcal{A}} - \Xi_{\mathcal{B}}\|_F }{\sqrt{\sum_{i=1}^n \|\xi^{(i)}\|_2^2}}\right|\geq t\right) 
    \leq 2\exp\left\{-\frac{c_4 \kappa^2 t^2}{64\sigma^4 \|\mathcal{A} - \mathcal{B}\|_{op}^2}\right\}.
\end{align*}
Thus, $\{X_{\mathcal{A}}\}_{\mathcal{A}}$ has sub-gaussian increments 
and there is a constant $c_5$ such that
\[\|X_{\mathcal{A}} - X_{\mathcal{B}}\|_{\psi_2} \leq \|X_{\mathcal{A}} - X_{\mathcal{B}} -  \frac{8\sigma}{\kappa}\|\mathcal{A} - \mathcal{B}\|_{op}\|_{\psi_2} +  \frac{8\sigma}{\kappa}\|\mathcal{A} - \mathcal{B}\|_{op} \leq  \frac{c_5 \sigma}{\kappa} \|\A - \mathcal{B}\|_{op}.\]
Then, by Theorem 8.1.6 in \cite{VershyninBook} and \eqref{e:cov_bnd}, 
\begin{align*}
    \EE\left[\sup_{P \in \mathcal{P}} \|P\ee\|\right] \leq \EE\left[|X_{\A}| \right] +  \EE\left[\sup_{\A \in B(\A_{*}, \kappa/4)}|X_{\A}| - \EE[|X_{\A}|] \right] 
    &\leq c_8m \sigma\sqrt{d+1}.
\end{align*}
Combining this bound with \eqref{e:unif_HW_2} gives 
\begin{align*}
     \PP\left(\sup_{P \in \mathcal{P}}\|P\ee\|^2  \geq (c_3\log(n)^2 + 1)m^2(d+1)\right) \leq e^{-c_{9} (d+1)\sigma^2m^2 \log(n) \min\left( \frac{\log(n)^2}{c^2_8}, 1\right)}.
\end{align*}
Taking $n \geq e^{-c_8^2}$ completes the proof.
\end{proof}

\end{document}